\documentclass[10pt]{article} %change fontsize globally between 10-12pt here%
\usepackage[utf8]{inputenc}
\usepackage[top=30mm, left=30mm, right=30mm, bottom=30mm]{geometry}
\usepackage{amsmath,amssymb,amsthm}
\usepackage{dsfont}
\usepackage{color}
\usepackage{lineno,hyperref}
\usepackage{yfonts}

\modulolinenumbers[5]

\newtheorem{theorem}{Theorem}[section]
\newtheorem{lem}[theorem]{Lemma}

\newtheorem{kor}[theorem]{Corollary}
\newtheorem{prop}[theorem]{Proposition}

\theoremstyle{definition}
\newtheorem{dfn}[theorem]{Definition}
\newtheorem{rem}[theorem]{Remark}

\definecolor{orange}{rgb}{1.0, 0.55, 0.0}

\DeclareMathOperator*{\divv}{div}

\begin{document}
	\title{Linearization and a superposition principle for deterministic and stochastic nonlinear Fokker-Planck-Kolmogorov equations}

	\author{Marco Rehmeier\footnote{Faculty of Mathematics, Bielefeld University, 33615 Bielefeld, Germany. E-Mail: mrehmeier@math.uni-bielefeld.de }}
	
	\date{}
	\maketitle
	\begin{abstract}
		We prove a superposition principle for nonlinear Fokker-Planck-Kolmogorov equations on Euclidean spaces and their corresponding linearized first-order continuity equation over the space of Borel (sub-)probability measures. As a consequence, we obtain equivalence of existence and uniqueness results for these equations. Moreover, we prove an analogous result for stochastically perturbed Fokker-Planck-Kolmogorov equations. To do so, we particularly show that such stochastic equations for measures are, similarly to the deterministic case, intrinsically related to linearized second-order equations on the space of Borel (sub-)probability measures.
	\end{abstract}
	
	
\section{Introduction}
In this work we are concerned with nonlinear Fokker-Planck-Kolmogorov equations (FPK-equations) on $\mathbb{R}^d$, both deterministic
\begin{equation}\label{NLFPKE}\tag{NLFPK}
\partial_t\mu_t = \mathcal{L}^*_{t,\mu_t}\mu_t, \,\, t \in [0,T],
\end{equation}and perturbed by a first-order stochastic term driven by a finite-dimensional Wiener process $W$
\begin{equation}\label{SNLFPKE}\tag{SNLFPK}
\partial_t\mu_t = \mathcal{L}^*_{t,\mu_t}\mu_t-\text{div}(\sigma(t,\mu_t))dW_t,\,\, t\in [0,T],
\end{equation}with solutions being continuous curves of subprobability measures $\mu_t \in \mathcal{SP}$. Here, $\mathcal{L}^*$ denotes the formal dual of a second-order differential operator acting on sufficiently smooth functions $\varphi: \mathbb{R}^d \to \mathbb{R}$ via
\begin{equation}\label{1}
\mathcal{L}_{t,\mu}\varphi(x) = \sum_{i,j=1}^{d}a_{ij}(t,\mu,x)\partial^2_{ij}\varphi(x)+\sum_{i=1}^{d}b_i(t,\mu,x)\partial_i\varphi(x)
\end{equation}with coefficients $a$ and $b$ depending on $(t,x) \in [0,T]\times \mathbb{R}^d$ and (in general non-locally) on the solution $\mu_t$. These equations are to be understood in distributional sense, see Definition \ref{Sol NLFPKE} and \ref{Def sol SNLFPKE}. The nonlinearity arises from the dependence of $\mathcal{L}$ and $\sigma$ on the solution itself, which renders the theory of existence and uniqueness of such equations significantly more difficult compared to the linear case. For a thorough introduction to the field, we refer to \cite{bogachev2015fokker} and the references therein. As shown in \cite{Rckner-lin.-paper}, the deterministic nonlinear equation (\ref{NLFPKE}) is naturally associated to a first-order linear continuity equation on $\mathcal{P}(\mathcal{SP})$, the space of Borel probability measures on $\mathcal{SP}$, of type
\begin{equation}\label{P-CE}\tag{$\mathcal{SP}$-CE}
\partial_t \Gamma_t = \mathbf{L}^*_t\Gamma_t, \,\, t\in [0,T],
\end{equation}in the sense of distributions, with the linear operator $\mathbf{L}$ acting on sufficiently smooth real functions on $\mathcal{SP}$ via the gradient operator $\nabla^{\mathcal{SP}}$ on $\mathcal{SP}$ as
$$\mathbf{L}_tF = \big \langle \nabla^{\mathcal{SP}}F, b_t+a_t\nabla \big \rangle_{L^2}.$$ Precise information on this operator and equation (\ref{P-CE}) are given in Section 3, in particular in Definition \ref{Def sol P-CE} and the paragraph  preceding it.\\
 Our first main result, Theorem \ref{main thm det case} states that each weakly continuous solution $(\Gamma_t)_{t \leq T}$ to (\ref{P-CE}) is a \textit{superposition} of solutions to (\ref{NLFPKE}), i.e. (denoting by $e_t$ the canonical projection $e_t: (\mu_t)_{t \leq T} \mapsto \mu_t$)
\begin{equation}\label{SuperPos eq}
\Gamma_t = \eta \circ e_t^{-1}
\end{equation}
for some probability measure $\eta$ concentrated on solution curves to (\ref{NLFPKE}) in a suitable sense.\\
 We also treat the stochastic case in a similar fashion. More precisely, in Section 4 we establish a new correspondence between the stochastic equation for measures (\ref{SNLFPKE}) and a corresponding second-order equation for curves $(\Gamma_t)_{t \leq T}$ in $\mathcal{P}(\mathcal{SP})$ of type
 \begin{equation}\label{P-FPKE}\tag{$\mathcal{SP}$-FPK}
 \partial_t \Gamma_t = (\mathbf{L}_t^{(2)})^*\Gamma_t,\,\, t \in [0,T],
 \end{equation}where, roughly, $$\mathbf{L}^{(2)}_t = \mathbf{L}_t+ \textit{ second-order perturbation}. $$ The second-order term stems from the stochastic perturbation of (\ref{SNLFPKE}) and will be geometrically interpreted in terms of a (formal) notion of the Levi-Civita connection on $\mathcal{SP}$. The second main result of this work, Theorem \ref{main thm stoch case}, is then the stochastic generalization of the deterministic case: For any solution $(\Gamma_t)_{t \leq T}$ to (\ref{P-FPKE}), there exists a solution process $(\mu_t)_{t \leq T}$ to (\ref{SNLFPKE}) on some probability space such that $\mu_t$ has distribution $\Gamma_t$. We stress that in both cases, we do not require any regularity of the coefficients.\\
 \\
 Let us embed these results into the general research in this direction. Let $b_t(\cdot): \mathbb{R}^d \to \mathbb{R}^d$ be an inhomogeneous vector field and consider the (nonlinear) ODE
 \begin{equation}\label{ODE intro}\tag{ODE}
 \frac{d}{dt}\gamma_t = b_t(\gamma_t), \,\, t \leq T
 \end{equation}and the linear continuity equation for curves of Borel (probability) measures on $\mathbb{R}^d$
 \begin{equation}\label{CE intro}\tag{CE}
 \partial_t \mu_t = -\divv (b_t\mu_t), \,\, t \leq T,
 \end{equation}understood in distributional sense. In the seminal paper \cite{Ambrosio2008}, L. Ambrosio showed the following: Any (probability) solution $(\mu_t)_{t \leq T}$ to (\ref{CE intro}) with an appropriate global integrability condition is a \textit{superposition} of solution curves to (\ref{ODE intro}), i.e. there exists a (probability) measure $\eta$ on the space of continuous paths with values in the state space of (\ref{ODE intro}), $C([0,T],\mathbb{R}^d)$, which is concentrated on solutions to (\ref{ODE intro}) such that
 $$\eta \circ e_t^{-1} = \mu_t, \,\, t\leq T.$$This allows to transfer existence and uniqueness results between the linear equation (\ref{CE intro}) and the nonlinear (\ref{ODE intro}). However, the linear equation must be studied on an infinite-dimensional space of (probability) measures. The analogy to our deterministic result from Section 3 is as follows: (\ref{ODE intro}) is replaced by (\ref{NLFPKE}), which, in spirit of this analogy, we interpret as a differential equation on the manifold-like state space $\mathcal{SP}$. Likewise, (\ref{CE intro}) is replaced by (\ref{P-CE}) and our first main result Theorem \ref{main thm det case} may be understood as the analogue of Ambrosio's result to the present setting. By passing from (\ref{NLFPKE}) to (\ref{P-CE}), we \textit{linearize} the equation.\\
 Concerning the stochastic case, consider a stochastic differential equation on $\mathbb{R}^d$
 \begin{equation}\label{SDE}\tag{SDE}
 dX_t = b(t,X_t)dt+\tilde{a}(t,X_t)dB_t, \,\, t\in [0,T].
 \end{equation} By Itô's formula, the one-dimensional marginals $\mu_t$ of any (weak) martingale solution $X$ solve the corresponding linear FPK-equation
 \begin{equation}\label{FPKE}\tag{FPK}
 \partial_t \mu_t = \mathcal{L}_{lin, t}^*\mu_t, \,\, t \in [0,T],
 \end{equation} where $\mathcal{L}_{lin}$ is a linear second-order diffusion operator with coefficients $b$ and $\frac{1}{2}\tilde{a}\tilde{a}^T$. Conversely, a superposition principle has successively been developed in increasingly general frameworks (cf. \cite{FIGALLI2008109, Kurtz2011, trevisan2016, Rckner-superpos_pr}): Under mild global integrability assumptions, for every weakly continuous solution curve of probability measures $(\mu_t)_{t \leq T}$ to (\ref{FPKE}), there exists a (weak) martingale solution $X$ to (\ref{SDE}) with one-dimensional marginals $(\mu_t)_{t \leq T}$, thereby providing an equivalence between solutions to (\ref{SDE}) and (\ref{FPKE}), which offers a bridge between probabilistic and analytic approaches to diffusion processes. As in the deterministic case, the transition from (\ref{SDE}) to (\ref{FPKE}) provides a \textit{linearization}, while at the same time it transfers the equation to a much higher dimensional state space. Concerning our stochastic result Theorem \ref{main thm stoch case}, we replace the stochastic equation on $\mathbb{R}^d$ by the stochastic equation for measures (\ref{SNLFPKE}) and the corresponding second-order equation for measures (\ref{FPKE}) by (\ref{P-FPKE}) and prove an analogous superposition result for solutions to the latter equation.\\
  The proofs of both the deterministic and stochastic result rely on superposition principles for differential equations on $\mathbb{R}^{\infty}$ and the corresponding continuity equation (for the deterministic case) and for martingale solutions and FPK-equations on $\mathbb{R}^{\infty}$ (for the stochastic case) by Ambrosio and Trevisan (\cite{ambrosio2014}, \cite{TrevisanPhD}). The key technique is to transfer (\ref{P-CE}) and (\ref{NLFPKE}) (and, similarly, (\ref{P-FPKE}) and (\ref{SNLFPKE}) for the stochastic case) to suitable equations on $\mathbb{R}^{\infty}$ via a homeomorphism between $\mathcal{SP}$ and $\mathbb{R}^{\infty}$ (replaced by $\ell^2$ for the stochastic case, in order to handle the stochastic integral).\\
 Moreover, our results also blend into the theory of \textit{distribution dependent} stochastic differential equations, also called \textit{McKean-Vlasov equations}, i.e. stochastic equations on Euclidean space of type
 \begin{equation}\label{DDSDE}\tag{DDSDE}
 dX_t = b(t,\mathcal{L}_{X_t},X_t)dt+\tilde{a}(t,\mathcal{L}_{X_t},X_t)dB_t,\,\, t\in [0,T],
 \end{equation}
 see the classical papers \cite{McKean1907, Funaki, scheutzow_1987} as well as the more recent works \cite{HUANG20194747, 1078-0947_2019_6_3017, Coghi2019StochasticNF}. Here, $\mathcal{L}_{X_t}$ denotes the distribution of $X_t$ and is not to be confused with the operators $\mathcal{L}_{t,\mu}$ and $\mathcal{L}_t$ from above. As in the non-distribution dependent case, where the curve of marginals of any solution to (\ref{SDE}) solves an equation of type (\ref{FPKE}), a similar observation holds here: Each solution $X$ to (\ref{DDSDE}) provides a solution to a nonlinear FPK-equation of type (\ref{NLFPKE}) via $\mu_t = \mathcal{L}_{X_t}$ and a corresponding superposition principle holds analogously to the linear case as well (\cite{barbu2020, doi:10.1137/17M1162780}).\\
However, while for (\ref{SDE}) the passage to (\ref{FPKE}) provides a complete linearization, the situation is different for equations of type (\ref{NLFPKE}). This stems from the observation that (\ref{DDSDE}) is an equation with two sources of nonlinearity. Hence, it seems natural to linearize (\ref{NLFPKE}) once more in order to obtain a linear equation, which is related to (\ref{DDSDE}) and (\ref{NLFPKE}) in a natural way. By the results of \cite{Rckner-lin.-paper}, this linear equation is of type (\ref{P-CE}). Similar considerations prevail in the stochastic case, where one considers equations of type (\ref{DDSDE}) with an additional source of randomness (we shall not pursue this direction in this work). \\
 \\
On the one hand, the superposition principles of Theorem \ref{main thm det case} and Theorem \ref{main thm stoch case} provide new structural results for nonlinear FPK-equations and its corresponding linearized equations on the space of probability measures over $\mathcal{SP}$, involving a geometric interpretation of the latter. On the other hand, it is our future plan to further study the geometry of $\mathcal{SP}$ as initiated in \cite{Rckner-lin.-paper} and this work to develop an analysis on such infinite-dimensional manifold-like spaces, which allows to solve linear equations of type (\ref{P-CE}) and (\ref{P-FPKE}) on such spaces. By means of the results of this work, one can then lift such solutions to solutions to the nonlinear equations for measures (\ref{NLFPKE}) and (\ref{SNLFPKE}), thereby obtaining new existence results for these nonlinear equations for measures.\\
 \\
 We point out that although our main aim is to lift weakly continuous solutions to (\ref{P-CE}) and (\ref{P-FPKE}) concentrated on probability measures to a measure on the space of continuous probability measure-valued paths $(\mu_t)_{t \leq T}$, for technical reasons we more generally develop our results for vaguely continuous subprobability solutions (i.e. $\mu_t \in \mathcal{SP}$). We comment on the advantages of this approach in Remark \ref{Rem explain why SP} for the deterministic case and note that similar arguments prevail in the stochastic case as well. However, due to the global integrability assumptions we consider, we are able to obtain results for probability solutions as desired.\\
 \\
 The organization of this paper is as follows. After introducing general notation and recalling basic properties of the spaces $\mathcal{P}$ and $\mathcal{SP}$ in Section 2, Section 3 contains the deterministic case, i.e. the superposition principle between solutions to (\ref{P-CE}) and (\ref{NLFPKE}). Here, the main result is Theorem \ref{main thm det case}. We use this result to prove an open conjecture of \cite{Rckner-lin.-paper} (cf. Proposition \ref{Prop conj Rckner}) and present several consequences. In Section 4, we treat the stochastic case for equations of type (\ref{SNLFPKE}), the main result being Theorem \ref{main thm stoch case}.
 \paragraph{Acknowledgements}
 Financial support by the German Science Foundation DFG (IRTG 2235) is gratefully
acknowledged.
 \section{Notation and Preliminaries}
 We introduce notation and repeat basic facts on spaces and topologies of measures.
 \subsection*{Notation}
 For a measure space $(\mathcal{X},\mathcal{A},\mu)$ and a measurable function $\varphi: \mathcal{X}\to \mathbb{R}$, we set $\mu(\varphi) := \int_{\mathcal{X}}\varphi(x)d\mu(x)$ whenever the integral is well-defined. For $x \in \mathcal{X}$, we denote by $\delta_x$ the \textit{Dirac measure in $x$}, i.e. $\delta_x(A) = 1$ if and only if $x \in A$ and $\delta_x(A) =0$ else. For a topological space $X$ with Borel $\sigma$-algebra $\mathcal{B}(X)$ we denote the set of continuous bounded functions by $C_b(X)$, the set of Borel probability measures on $X$ by $\mathcal{P}(X)$ and write $\mathcal{P} = \mathcal{P}(\mathbb{R}^d)$. If $Y \in \mathcal{B}(X)$, we let $\mathcal{B}(X)_{\upharpoonright Y}$ denote the \textit{trace of $Y$on} $\mathcal{B}(X)$. For $T>0$, a family $(\mu_t)_{t \leq T} = (\mu_t)_{t \in [0,T]}$ of finite Borel measures on $\mathbb{R}^d$ is a \textit{Borel curve}, if $t \mapsto \mu_t(A)$ is Borel measurable for each $A \in \mathcal{B}(\mathbb{R}^d)$. A set of functions $\mathcal{G} \subseteq C_b(\mathbb{R}^d)$ is called \textit{measure-determining}, if $\mu(g) = \nu(g)$ for each $g \in \mathcal{G}$ implies $\mu = \nu$ for any two finite Borel measures $\mu, \nu$ on $\mathbb{R}^d$.\\
 For $x,y \in \mathbb{R}^d$, the usual inner product is denoted by $x \cdot y$ and, with slight abuse of notation, we also denote by $x \cdot y = \sum_{k \geq 1 }x_ky_k$ the inner product in $\ell^2$ (the Hilbert space of square-summable real-valued sequences $x=(x_k)_{k \geq1}$). For $\varphi \in C_b(\mathbb{R}^d)$, we set $||\varphi||_{\infty} := \underset{x \in \mathbb{R}^d}{\text{sup}}|\varphi(x)|$. If $\varphi$ has first- and second-order partial derivatives, we denote them by $\partial_i\varphi$ and $\partial^2_{ij}\varphi$ for $i,j \leq d$.\\
 \\
 We use notation for function spaces as follows. For $k \in \mathbb{N}_0$, $C^k_b(\mathbb{R}^d)$ denotes the subset of functions $\varphi$ in $C_b(\mathbb{R}^d)$ with continuous, bounded partial derivatives up to order $k$, with the usual norm $||\varphi||_{C^2_b} = \text{max}(||\varphi||_{\infty},||\partial_i\varphi||_{\infty}, ||\partial^2_{ij}\varphi||_{\infty})$ for $k = 2$. Likewise, $C^k_c(\mathbb{R}^d)$ denotes the subset of all such $\varphi$ with compact support; for $k = 0$, we write $C_c(\mathbb{R}^d)$ instead. For $n \geq 1$, $p \geq 1$ and a measure $\mu$ on $\mathcal{B}(\mathbb{R}^d)$, we denote by $L^p(\mathbb{R}^d,\mathbb{R}^n;\mu)$ the space of Borel functions $\varphi: \mathbb{R}^d \to \mathbb{R}^n$ such that
 $$\int_{\mathbb{R}^d}||\varphi(x)||^pd\mu(x) < +\infty,$$
 where $||\cdot||$ denotes the standard Euclidean norm on $\mathbb{R}^n$. For $p =2$, $\langle \cdot, \cdot \rangle_{L^2(\mathbb{R}^d,\mathbb{R}^n;\mu)}$ denotes the usual inner product on the Hilbert space $L^2(\mathbb{R}^d,\mathbb{R}^n;\mu)$. For  $T >0$ and a topological space $Y$, we write $C_TY$ for the set of continuous functions $\varphi:[0,T]\to Y$. By $\mathbb{S}^+_d$ we denote the space of symmetric, positive-semidefinite $d\times d$-matrices with real entries.
 \subsection*{Basic properties of spaces of measures}
 \subsubsection*{Probability measures}
 For a topological space $X$, we endow $\mathcal{P}(X)$ with the topology of weak convergence of measures, i.e. the initial topology of the maps $\mu \mapsto \mu(\varphi)$, $\varphi \in C_b(X)$. If $X$ is Polish, then so is $\mathcal{P}(X)$.
\subsubsection*{Subprobability measures}
 By $\mathcal{SP}$ we denote the set of all Borel subprobability measures on $\mathbb{R}^d$, i.e. $\mu \in \mathcal{SP}$ if and only if $\mu$ is a non-negative measure on $\mathcal{B}(\mathbb{R}^d)$ with $\mu(\mathbb{R}^d) \leq 1$. Throughout, we endow $\mathcal{SP}$ with the \textit{vague topology}, i.e. the initial topology of the maps $\mu \mapsto \mu(g)$, $g \in C_c(\mathbb{R}^d)$. Hence, a sequence $(\mu_n)_{n \geq 1}$ converges to $\mu$ in $\mathcal{SP}$ if and only if $\mu_n(g) \underset{n \to \infty}{\longrightarrow} \mu(g)$ for each $g \in C_c(\mathbb{R}^d)$. Its Borel $\sigma$-algebra is denoted by $\mathcal{B}(\mathcal{SP})$. In particular, $\mathcal{P}(\mathcal{SP})$, the set of Borel probability measures on $\mathcal{SP}$, is a topological space with the weak topology of probability measures on $(\mathcal{SP},\mathcal{B}(\mathcal{SP}))$. The Riesz-Markov representation theorem yields that $\mathcal{SP}$ with the vague topology coincides with the positive half of the closed unit ball of the dual space of $C_c(\mathbb{R}^d)$ with the weak*-topology. Hence $\mathcal{SP}$ with the vague topology is compact. It is also Polish and $\mu \mapsto \mu(\mathbb{R}^d)$ is vaguely lower semicontinuous, see \cite[Ch.4.1]{OK}. In particular,
 $\mathcal{P}  \in \mathcal{B}(\mathcal{SP}).$
Recall that $\mathcal{B}(\mathcal{P}) = \mathcal{B}(\mathcal{SP})_{\upharpoonright_{\mathcal{P}}}$. Hence, in the sequel we may consider measures $\Gamma \in \mathcal{P}(\mathcal{P})$ as elements in $\mathcal{P}(\mathcal{SP})$ with mass on $\mathcal{P}$.\\
 In contrast to weak convergence in $\mathcal{P}$, vague convergence in $\mathcal{SP}$ can be characterized by countably many functions in a sense made precise by Lemma \ref{Prop G early}. The fact that this is not true for weak convergence in $\mathcal{P}$ is the main reason why we formulate all equations for subprobability measures, although we are mainly interested in the case of probability solutions. More details in this direction are stated in Remark \ref{Rem explain why SP}.
 
 \section{Superposition Principle for deterministic nonlinear Fokker-Planck-Kolmogorov Equations}

 Fix $ T>0$ throughout, let each component of the coefficients
 $$a = (a_{ij})_{i,j \leq d}: [0,T]\times \mathcal{SP}\times \mathbb{R}^d \to \mathbb{S}^+_d, \, b = (b_i)_{i \leq d}: [0,T]\times \mathcal{SP}\times \mathbb{R}^d \to \mathbb{R}^d$$
 be $\mathcal{B}([0,T])\otimes \mathcal{B}(\mathcal{SP})\otimes \mathcal{B}(\mathbb{R}^d) / \mathcal{B}(\mathbb{R})$-measurable and consider the operator $\mathcal{L}_{t,\mu}$ as in (\ref{1}). 
 \begin{dfn}\label{Sol NLFPKE}
 \begin{enumerate}
 	\item [(i)] A vaguely continuous curve $(\mu_t)_{t \leq T} \subseteq \mathcal{SP}$ is a \textit{subprobability solution to }(\ref{NLFPKE}), if for each $i, j \leq d$ the global integrability condition
 		\begin{equation}\label{2}
 		\int_0^T \int_{\mathbb{R}^d}|a_{ij}(t,\mu_t,x)|+|b_i(t,\mu_t,x)|d\mu_t(x)dt < +\infty
 		\end{equation}holds and for each $\varphi \in C^2_c(\mathbb{R}^d)$ and $t \in [0,T]$
 		\begin{equation}\label{3}
 		\int_{\mathbb{R}^d}\varphi(x)d\mu_t(x)-\int_{\mathbb{R}^d}\varphi(x)d\mu_0(x) = \int_0^t\int_{\mathbb{R}^d}\mathcal{L}_{s,\mu_s}\varphi(x)d\mu_s(x)ds.
 		\end{equation}
 
 	\item[(ii)] A \textit{probability solution} to (\ref{NLFPKE}) is a curve $(\mu_t)_{t \leq T}\subseteq \mathcal{P}$ fulfilling (\ref{2}) and (\ref{3}) such that $t \mapsto \mu_t$ is weakly continuous.
 \end{enumerate}
 \end{dfn}
 Since vaguely continuous curves of measures are in particular Borel curves, all integrals in the above definition are defined. Below we shortly refer to \textit{subprobability }and \textit{probability solutions} and keep in mind the respective continuity conditions. In the literature, more general notions of solutions to (\ref{NLFPKE}) are considered, such as (possibly discontinuous) curves of signed, bounded measures \cite{bogachev2015fokker}. However, in this work, we restrict attention to continuous (sub-)probability solutions. In presence of the global integrability condition (\ref{2}), we make the following observation.
 \begin{rem}\label{Rem mass conserv}
 	\begin{enumerate}
 		\item [(i)] Any subprobability solution $(\mu_t)_{t \leq T}$ with $\mu_0 \in \mathcal{P}$ is a probability solution. Indeed, to prove this it suffices to show $\mu_t(\mathbb{R}^d) =1$ for each $t \leq T$. Since $(\mu_t)_{t \leq T}$ fulfills (\ref{3}), it suffices to choose a sequence $\varphi_l$, $l \geq 1$, from $C^2_c(\mathbb{R}^d)$ with the following properties: $0 \leq \varphi_l \nearrow 1$ pointwise such that $\partial_i \varphi_l \underset{l \to \infty}{\longrightarrow} 0$, $\partial^2_{ij}\varphi_l \underset{l \to \infty}{\longrightarrow}0$ pointwise with all first and second order derivatives bounded by some $M < +\infty$ uniformly in $l \geq 1$ and $x \in \mathbb{R}^d$. Considering (\ref{3}) for the limit $l \to \infty$, we obtain, by (\ref{2}) and dominated convergence, for each $t \in [0,T]$
 		$$\int_{\mathbb{R}^d}1 d\mu_t - \int_{\mathbb{R}^d}1 d\mu_0 = 0$$
 	and hence the claim.
 		\item[(ii)] By the above argument, one shows that for any subprobability solution, (\ref{3}) holds for each $\varphi \in C^2_b(\mathbb{R}^d)$. 
 	\end{enumerate} 
  \end{rem}
  \subsubsection*{Geometric approach to $\mathbf{\mathcal{SP}}$}For our goals, it is preferable to consider $\mathcal{SP}$ as a manifold-like space. We refer the reader to the appendix in \cite{Rckner-lin.-paper}, where for the space of probability measures $\mathcal{P}$ the tangent spaces $T_{\mu}\mathcal{P} = L^2(\mathbb{R}^d,\mathbb{R}^d;\mu)$ and a suitable test function class $\mathcal{F}C^2_b(\mathcal{P})$, 
  \begin{equation}\label{Test fct. Rckner}
  F \in \mathcal{F}C^2_b(\mathcal{P}) \iff F: \mu \mapsto f\big(\mu(\varphi_1),\dots,\mu(\varphi_n)\big) \text{ for }n \geq 1, f\in C^1_b(\mathbb{R}^n), \, \varphi_i \in C^{\infty}_c(\mathbb{R}^d),
  \end{equation}
  have been introduced. Further, based on these choices, a natural pointwise definition of the gradient $\nabla^{\mathcal{P}}F$ as a section in the tangent bundle $$T\mathcal{P} = \bigsqcup_{\mu \in \mathcal{P}}T_{\mu}\mathcal{P}$$for $F$ as above is given by
  $$\nabla^{\mathcal{P}}F(\mu) := \sum_{k=1}^{n}\partial_kf\big(\mu(\varphi_1),\dots,\mu(\varphi_n)\big)\nabla \varphi_k \in T_{\mu}\mathcal{P},$$which is shown to be independent of the representation of $F$ in terms of $f$ and $\varphi_i$. The setting in the present paper is nearly identical, but we consider the manifold-like space $\mathcal{SP}$ with the vague topology instead of $\mathcal{P}$ with the weak topology as in \cite{Rckner-lin.-paper}, because $\mathcal{SP}$ is embedded in $\mathbb{R}^{\infty}$ in the following sense. Let
  \begin{equation}\label{Fct.class G}
  \mathcal{G} = \{g_i, i \geq 1\}
  \end{equation} be dense in $C^2_c(\mathbb{R}^d)$ with respect to $||\cdot||_{C^2_b}$ such that no $g_i$ is constantly $0$. Clearly, any such set of functions is dense in $C_c(\mathbb{R}^d)$ with respect to uniform convergence and measure-determining. Such sets of functions are sufficiently extensive to characterize the topology of $\mathcal{SP}$ as well as solutions to (\ref{NLFPKE}):
 	
 	\begin{lem}\label{Prop G early}
 		Let $\mathcal{G}$ be any set of functions with the properties mentioned above and let $(\mu_n)_{n \geq 1} \subseteq \mathcal{SP}$. Then,
 		\begin{enumerate}
 			\item [(i)] $(\mu_n)_{n \geq 1}$ converges vaguely to $\mu \in \mathcal{SP}$ if and only if
 			\begin{equation*}
 			\mu_n(g_i) \underset{n \to \infty}{\longrightarrow}\mu(g_i)
 			\end{equation*}
 			for each $g_i \in \mathcal{G}$.
 			\item[(ii)] A vaguely continuous curve $(\mu_t)_{t \leq T} \subseteq \mathcal{SP}$, which fulfills (\ref{2}), is a subprobability solution to (\ref{NLFPKE}) if and only if (\ref{3}) holds for each $g_i \in \mathcal{G}$ in place of $\varphi$.
 		\end{enumerate}
 	\end{lem}
 	\begin{proof}
 		\begin{enumerate}
 			\item [(i)] From $\mu_n(g_i) \underset{n \to \infty}{\longrightarrow}\mu(g_i)$ for each $g_i \in \mathcal{G}$, one obtains for each $f \in C_c(\mathbb{R}^d)$ and $\epsilon>0$ by choosing $g_i \in \mathcal{G}$ with $||f-g_i||_{\infty} < \frac{\epsilon}{3}$
 			\begin{equation}\label{4}
 			|\mu_n(f)-\mu(f)| \leq |\mu_n(f)-\mu_n(g_i)|+|\mu_n(g_i)-\mu(g_i)|+|\mu(g_i)-\mu(f)| \leq \epsilon
 			\end{equation}for all sufficiently large $n \geq 1$.
 			\item[(ii)] Let $\varphi \in C^2_c(\mathbb{R}^d)$ be approximated uniformly up to second-order derivatives by a sequence $\{g_{i_k}\}_{k \geq 1}$ from $\mathcal{G}$. Considering (\ref{3}) for such $g_{i_k}$ and letting $k \to \infty$, the result follows by dominated convergence, which applies due to (\ref{2}).
 		\end{enumerate}
 	\end{proof}
 Considering $\mathcal{SP}$ as a (infinite-dimensional) manifold-like topological space, any set of functions $\mathcal{G}$ as above provides a global chart (i.e., an atlas consisting of a single chart) for $\mathcal{SP}$, as it yields an embedding $\mathcal{SP} \subseteq \mathbb{R}^{\infty}$ (cf. Lemma \ref{Lem aux G, J}).\\
  Consider $\mathbb{R}^{\infty}$ as a Polish space with the topology of pointwise convergence and the range $G(\mathcal{SP}) \subseteq \mathbb{R}^{\infty}$ of $G$ as introduced below with its subspace topology. We write $C_TG(\mathcal{SP})$ for the set of all elements in $C_T\mathbb{R}^{\infty}$ with values in $G(\mathcal{SP})$. For $u \in [0,T]$, we denote by $e_u$ the canonical projection on $C_T\mathcal{SP}$
 $$e_u: (\mu_t)_{t\leq T} \mapsto \mu_u$$
 and, likewise, by $e^{\infty}_u$ the projection on $C_T\mathbb{R}^{\infty}$. Subsequently, without further mentioning, we consider the spaces $C_T\mathcal{SP}$ and $C_T\mathbb{R}^{\infty}$ with $\sigma$-algebras
 $$\mathcal{B}(C_T\mathcal{SP}) = \sigma(e_t, t\in [0,T]) \text{ and }\mathcal{B}(C_T\mathbb{R}^{\infty}) = \sigma(e_t^{\infty}, t \in [0,T]),$$
 respectively. These algebras coincide with the Borel $\sigma$-algebras with respect to the topology of uniform convergence (because both $\mathcal{SP}$ and $\mathbb{R}^{\infty}$ are Polish). Also, consider $C_TG(\mathcal{SP})$ with the natural subspace $\sigma$-algebra of $\mathcal{B}(C_T\mathbb{R}^{\infty})$. We refer to these $\sigma$-algebras as the \textit{canonical $\sigma$-algebras} on the respective spaces and denote the set of probability measures on the respective $\sigma$-algebras by $\mathcal{P}(C_T\mathcal{SP})$ and $\mathcal{P}(C_T\mathbb{R}^{\infty})$.
 \begin{lem}\label{Lem aux G, J}
 	Let $\mathcal{G} = \{g_i\}_{i \geq 1} $ be a set of functions as in (\ref{Fct.class G}).
 	\begin{enumerate}
 		\item [(i)] The map $G$, depending on $\mathcal{G}$, 
 		\begin{equation}\label{Def map G}
 		G : \mathcal{SP} \to \mathbb{R}^{\infty}, G(\mu) := (\mu(g_i))_{i \geq 1}
 		\end{equation}
 		is a homeomorphism between $\mathcal{SP}$ and its range $G(\mathcal{SP})$ (hence, formally, a \textup{global chart} for $\mathcal{SP}$). In particular, $G(\mathcal{SP}) \subseteq \mathbb{R}^{\infty}$ is compact. Moreover, if $\mathcal{G}' = \{g_i', i \geq 1\}$ is another set as in (\ref{Fct.class G}) with corresponding chart $G'$, then $G' = G \circ \mathcal{V}$ for a unique homeomorphism $\mathcal{V}$ on $\mathcal{SP}$.
 		\item[(ii)] The map
 		$$J : C_T\mathcal{SP} \to C_T\mathbb{R}^{\infty},\, J((\mu_t)_{t \leq T}) := G(\mu_t)_{t \leq T}$$
 		is measurable and one-to-one with measurable inverse $J^{-1}: C_TG(\mathcal{SP}) \to C_T\mathcal{SP}$. Further, $C_TG(\mathcal{SP}) \subseteq C_T\mathbb{R}^{\infty}$ is a measurable set, i.e. $C_TG(\mathcal{SP}) \subseteq \mathcal{B}(C_T\mathbb{R}^{\infty})$.
 	\end{enumerate}
 \end{lem}
 \begin{proof}
 	\begin{enumerate}
 		\item [(i)] The continuity of $G$ is obvious by definition of the vague topology on $\mathcal{SP}$ and since $\mathcal{G} \subseteq C_c(\mathbb{R}^d)$. Since $\mathcal{SP}$ is compact with respect to the vague topology, compactness of $G(\mathcal{SP}) \subseteq \mathbb{R}^{\infty}$ follows. $\mathcal{G}$ is measure-determining on $\mathbb{R}^d$, which implies that $G$ is one-to-one. Since by definition 
 		$$G(\mu_n) \underset{n \to \infty}{\longrightarrow} G(\mu) \iff \mu_n(g_i) \underset{n \to \infty}{\longrightarrow} \mu(g_i) \text{ for each }g_i \in \mathcal{G},$$ 
 		continuity of $G^{-1}$ follows from Lemma \ref{Prop G early} (i). The final assertion follows, since for $G'$ as in the assertion, $\mathcal{V}:= G^{-1}\circ \mathcal{W}\circ G'$ with $\mathcal{W}:G'(\mathcal{SP}) \to G(\mathcal{SP})$, $\mathcal{W}: (\mu(g'_i))_{i \geq 1} \mapsto (\mu(g_i))_{i \geq 1}$ is a homeomorphism.
 		\item[(ii)] Since $G$ is one-to-one and measurable, so is $J$. Clearly, $C_TG(\mathcal{SP})$ is the range of $J$ and hence $J: C_T\mathcal{SP} \to C_TG(\mathcal{SP})$ is a bijection between standard Borel spaces (the latter, because $\mathcal{SP}$ and $G(\mathcal{SP})$ with the respective topologies are Polish). This yields the measurability of $J^{-1}$. Finally, closedness of $G(\mathcal{SP}) \subseteq \mathbb{R}^{\infty}$ implies that $C_TG(\mathcal{SP})\subseteq C_T\mathbb{R}^{\infty}$ is a measurable set, because $G(\mathcal{SP})$ carries the subspace topology inherited from $\mathbb{R}^{\infty}$. 
 	\end{enumerate}
 \end{proof}
 By part (i) of the previous lemma it is justified to fix a set $\mathcal{G} = \{g_i, i \geq 1\}$ for the remainder of the section. In order to switch between test functions on $\mathcal{SP}$ and $\mathbb{R}^{\infty}$ in an equivalent way, we slightly deviate from the test function class presented in \cite{Rckner-lin.-paper} (see (\ref{Test fct. Rckner})) and, instead, consider
 $$\mathcal{F}C^2_b(\mathcal{G}) := \{F : \mathcal{SP}\to \mathbb{R} | F(\mu) =  f\big(\mu(g_1),\dots,\mu(g_n)\big), f \in C^2_b(\mathbb{R}^n), n \geq 1 \},$$where the restriction $f \in C^2_b(\mathbb{R}^n)$ is made for consistency with the stochastic case later on only. We summarize our geometric interpretation of $\mathcal{SP}$, which is of course still a close adaption of the ideas presented in \cite{Rckner-lin.-paper}:\\
 For the manifold-like space $\mathcal{SP}$, we consider smooth test functions $F \in \mathcal{F}C^2_b(\mathcal{G})$, with $\mathcal{G}$ being fixed as in (\ref{Fct.class G}). For each $\mu \in \mathcal{SP}$, we have the tangent space $T_{\mu}\mathcal{SP} = L^2(\mathbb{R}^d,\mathbb{R}^d;\mu)$ and the gradient
 $$\nabla ^{\mathcal{SP}}F(\mu) = \sum_{k=1}^{n}\partial_kf\big(\mu(g_1),\dots,\mu(g_n)\big)\nabla g_k \in T_{\mu}\mathcal{SP}$$ for $\mathcal{F}C^2_b(\mathcal{G}) \ni F: \mu \mapsto f\big(\mu(g_1),\dots,\mu(g_n)\big)$ as a section in the tangent bundle $T\mathcal{SP}$, which is independent of the representation of $F$. Adding to the approach of $\mathcal{SP}$ as a manifold-like space, the global chart $G$ as in (\ref{Def map G}) embeds $\mathcal{SP}$ into $\mathbb{R}^{\infty}$. However, we do not rigorously treat $\mathcal{SP}$ as a (Fréchet-)manifold and consider the embedding $\mathcal{SP} \subseteq \mathbb{R}^{\infty}$ merely as a tool to transfer (\ref{NLFPKE}) and its corresponding continuity equation to equivalent equations over $\mathbb{R}^{\infty}$, as outlined below. 
\subsubsection*{The continuity equation (\ref{P-CE})}
 As mentioned in the introduction, we study the linear continuity equation associated to (\ref{NLFPKE}) as derived in \cite{Rckner-lin.-paper}, which is a first-order equation for curves of measures on $\mathcal{SP}$. More precisely, in analogy to the derivation in \cite{Rckner-lin.-paper}, it is readily seen that any subprobability solution $(\mu_t)_{t \leq T}$ to (\ref{NLFPKE}) induces a curve of elements in $\mathcal{P}(\mathcal{SP})$, $\Gamma_t := \delta_{\mu_t}$, $t \leq T$, with
 \begin{equation}\label{Solved eq1}
 \int_{\mathcal{SP}}F(\mu)d\Gamma_t(\mu)- \int_{\mathcal{SP}}F(\mu)d\Gamma_0(\mu) = \int_0^t \int_{\mathcal{SP}}\big \langle \nabla^{\mathcal{SP}}F(\mu), b(s,\mu)+a(s,\mu)\nabla\big \rangle_{L^2(\mu)}d\Gamma_s(\mu)ds
 \end{equation}for each $t \leq T$ and $F \in \mathcal{F}C^2_b(\mathcal{G})$. Here, we set $b(s,\mu) = b(s,\mu,\cdot): \mathbb{R}^d \to \mathbb{R}^d$ (similarly for $a(s,\mu)$), $L^2(\mu) = L^2(\mathbb{R}^d,\mathbb{R}^d;\mu)$ and abbreviated
 $$\big \langle \nabla^{\mathcal{SP}}F(\mu), b(s,\mu)+a(s,\mu)\nabla\big \rangle_{L^2(\mu)} = \int_{\mathbb{R}^d}\sum_{k=1}^{n}(\partial_kf)\big(\mu(g_1),\dots,\mu(g_n)\big) a_{ij}(s,\mu,x)\partial^2_{ij}g_k(x)+b_i(s,\mu,x)\partial_i g_k(x)d\mu(x).$$
 We rewrite (\ref{Solved eq1}) in distributional form in duality with $\mathcal{F}C^2_b(\mathcal{G})$ as
 \begin{equation*}
 \partial_t \Gamma_t = -\nabla^{\mathcal{SP}}\cdot([b_t+a_t\nabla]\Gamma_t), \,\, t \leq T.
 \end{equation*}Setting
 \begin{equation}\label{5}
 \mathbf{L}_tF(\mu) := \big\langle a(t,\mu)\nabla+b(t,\mu), \nabla^{\mathcal{SP}}F(\mu)\big\rangle_{L^2(\mu)},
 \end{equation}this is just the linear continuity equation (\ref{P-CE}). The term $a\nabla$ has rigorous meaning only, if $a$ has sufficiently regular components in order to put the derivative $\nabla$ on $a$ via integration by parts, which we do not assume at any point.
 Considering $\mathcal{SP}$ as a manifold-like space, one may formally regard to $a\nabla + b$ as a time-dependent section in the tangent bundle $T\mathcal{SP}$.\\
 More generally, we introduce the following notion of solution to (\ref{P-CE}) (see \cite{Rckner-lin.-paper}):
 \begin{dfn}\label{Def sol P-CE}
 	A weakly continuous curve $(\Gamma_t)_{t \leq T} \subseteq \mathcal{P}(\mathcal{SP})$ is a \textit{solution to }(\ref{P-CE}), if the integrability condition
 		\begin{equation}\label{aux_revised1}
 		\int_0^T\int_{\mathcal{SP}}||b(t,\mu,\cdot)||_{L^1(\mathbb{R}^d,\mathbb{R}^d;\mu)}+||a(t,\mu,\cdot)||_{L^1(\mathbb{R}^d,\mathbb{R}^{d^2};\mu)}d \Gamma_t(\mu)dt < +\infty
 		\end{equation} is fulfilled and for each $F \in \mathcal{F}C^2_b(\mathcal{G})$ and $t \in [0,T]$
 		\begin{equation}\label{6}
 		\int_{\mathcal{SP}}F(\mu)d\Gamma_t(\mu)- \int_{\mathcal{SP}}F(\mu)d\Gamma_0(\mu) = \int_0^t \int_{\mathcal{SP}}\mathbf{L}_sF(\mu)d\Gamma_s(\mu)ds
 		\end{equation}holds (which is just (\ref{Solved eq1})).

 \end{dfn}
 The choice of $\mathcal{G}$ as in (\ref{Fct.class G}) implies that any solution in the above sense fulfills (\ref{6}) even for each $F \in \mathcal{F}C^2_b(\mathcal{SP})$, i.e. for the larger class of test functions considered in \cite{Rckner-lin.-paper} (upon extending their domain from $\mathcal{P}$ to $\mathcal{SP}$). In particular, this notion of solution is independent of $\mathcal{G}$. The main result of this chapter, Theorem \ref{main thm det case}, states that any solution to (\ref{P-CE}) as in Definition \ref{Def sol P-CE} arises as a superposition of solutions to (\ref{NLFPKE}). Note that for $\nu \in \mathcal{SP}$, uniqueness of solutions $(\Gamma_t)_{t \leq T}$ to (\ref{P-CE}) with $\Gamma_0 = \delta_{\nu}$ implies uniqueness of subprobability solutions $(\mu_t)_{t \leq T}$ to (\ref{NLFPKE}) with $\mu_0 = \nu$.
 
 \subsubsection*{Transferring (\ref{NLFPKE}) and (\ref{P-CE}) to $\mathbb{R}^{\infty}$}
 We use the global chart $G: \mathcal{SP} \to \mathbb{R}^{\infty}$ and the map $J$ of Lemma \ref{Lem aux G, J} to reformulate both (\ref{NLFPKE}) and (\ref{P-CE}) on $\mathbb{R}^{\infty}$. Define a Borel vector field $\bar{B} = (\bar{B}_k)_{k \in \mathbb{N}}$ component-wise as follows. For $t \in [0,T]$, consider the Borel set $A_t \in \mathcal{B}(\mathcal{SP})$,
 $$A_t := \bigg\{\mu \in \mathcal{SP}: \int_{\mathbb{R}^d}|a_{ij}(t,\mu,x)+|b_i(t,\mu,x)|d\mu(x)< \infty \,\, \forall i,j \leq d\bigg\}$$
 and define $B := (B_k)_{k \in \mathbb{N}}$ via
 $$B_k(t,\mu):= \int_{\mathbb{R}^d}\mathcal{L}_{t,\mu}g_k(x)d\mu(x),\quad (t,\mu) \in [0,T]\times A_t.$$
 Now define $\bar{B}: [0,T]\times \mathbb{R}^\infty \to \mathbb{R}^\infty$ via
 $$\bar{B}(t,z):= \begin{cases}
 	B(t,G^{-1}(z)),&\quad \text{ if }z \in G(A_t)\\
 	0,&\quad \text{ else,}
 \end{cases}$$
which is Borel measurable by Lemma \ref{Lem aux G, J}.  Next, consider the differential equation on $\mathbb{R}^{\infty}$
 \begin{equation}\label{Rinfty-ODE}\tag{$\mathbb{R}^{\infty}$-ODE}
 \frac{d}{dt} z_t = \bar{B}(t,z_t), \,\,t \in [0,T],
 \end{equation}which turns out to be the suitable analogue to (\ref{NLFPKE}) on $\mathbb{R}^{\infty}$. Analogously, the corresponding continuity equation for curves of Borel probability measures $\bar{\Gamma}_t$ on $\mathbb{R}^{\infty}$, i.e.
 \begin{equation}\label{Rinfty-CE}\tag{$\mathbb{R}^{\infty}$-CE}
 \partial_t \bar{\Gamma}_t = -\bar{\nabla}\cdot(\bar{B}\bar{\Gamma}_t), \,\, t\in [0,T],
 \end{equation}with $\bar{\nabla}$ as introduced below, is the natural analogue of the linear continuity equation (\ref{P-CE}). Roughly, these analogies are to be understood in the sense that solutions to (\ref{NLFPKE}) and (\ref{P-CE}) can be transferred to solutions to (\ref{Rinfty-ODE}) and (\ref{Rinfty-CE}), respectively, via the chart $G$. We refer to the proof of the main result below for more details. Let $$p_i: z \mapsto z_i,\,\, z \in \mathbb{R}^{\infty}$$ denote the canonical projection to the $i$-th component, set $\pi_n = (p_1,\dots,p_n)$ and 
 $$\mathcal{F}C^2_b(\mathbb{R}^{\infty}) := \{\bar{F}: \mathbb{R}^{\infty} \to \mathbb{R} | \bar{F}= f \circ \pi_n, f \in C^2_b(\mathbb{R}^n), n \geq 1\}.$$
 By $\bar{\nabla}$ we denote the gradient-type operator on $\mathbb{R}^{\infty}$, acting on $\bar{F} = f \circ \pi_n \in \mathcal{F}C^2_b(\mathbb{R}^{\infty})$ via
 \begin{equation}\label{Def nabla gradient}
 \bar{\nabla}\bar{F}(z):= \big((\partial_1f)(\pi_nz),\dots,(\partial_nf) (\pi_nz),0,0,\dots\big).
 \end{equation}Again, the restriction to test functions possessing second-order derivatives is made in order to be consistent with the stochastic (second-order) case later on.
 \begin{dfn}\label{Def sol Rinfty-eq}
 	\begin{enumerate}
 		\item [(i)] A curve $(z_t)_{t \leq T} = ((p_i\circ z_t)_{i \geq 1})_{t \leq T} \in C_T\mathbb{R}^{\infty}$ is a \textit{solution to} (\ref{Rinfty-ODE}), if for each $i \geq 1$ the $\mathbb{R}$-valued curve $t \mapsto p_i \circ z_t$ is absolutely continuous with weak derivative $t \mapsto p_i \circ \bar{B}(t,z_t)$ $dt$-a.s.
 		\item[(ii)] A curve $(\bar{\Gamma}_t)_{t \leq T} \subseteq \mathcal{P}(\mathbb{R}^{\infty})$ is a \textit{solution to }(\ref{Rinfty-CE}), if it is weakly continuous, fulfills the integrability condition
 		\begin{equation}\label{8}
 		\int_0^T \int_{\mathbb{R}^{\infty}} |\bar{B}_k(t,z)|d\bar{\Gamma}_t(z)dt < +\infty \text{ for each }k \geq 1
 		\end{equation} and for each $\bar{F}\in \mathcal{F}C^2_b(\mathbb{R}^{\infty})$ the identity
 		$$\int_{\mathbb{R}^{\infty}}\bar{F}(z)d\bar{\Gamma}_t(z)- \int_{\mathbb{R}^{\infty}}\bar{F}(z)d\bar{\Gamma}_0(z) = \int_0^t \int_{\mathbb{R}^{\infty}}\bar{\nabla} \bar{F}(z) \cdot \bar{B}(s,z)d\bar{\Gamma}_s(z)ds$$
 		holds for all $t \in [0,T]$.
 	\end{enumerate}
 \end{dfn}
 \subsection{Main Result: Deterministic case}
The following theorem is the main result for the deterministic case.
 \begin{theorem}\label{main thm det case}
 	Let $a,b$ be Borel coefficients on $[0,T]\times \mathcal{SP}\times \mathbb{R}^d$. For any weakly continuous solution $(\Gamma_t)_{t \leq T}$ to (\ref{P-CE}) in the sense of Definition \ref{Def sol P-CE}, there exists a probability measure $\eta \in \mathcal{P}(C_T\mathcal{SP})$, which is concentrated on vaguely continuous subprobability solutions to (\ref{NLFPKE}) such that
 	$$\eta \circ e_t^{-1} = \Gamma_t,\,\, t \in [0,T].$$\\
 	Moreover, if $\Gamma_0 \in \mathcal{P}(\mathcal{P})$, then $\eta$ is concentrated on weakly continuous probability solutions to (\ref{NLFPKE}).
 \end{theorem}
 
 The proof relies on a superposition principle for measure-valued solution curves of continuity equations on $\mathbb{R}^{\infty}$ and its corresponding differential equation, which we recall in Proposition \ref{Sp-pr. Rinfty prop} below. More precisely, we proceed in three steps. First, we transfer $(\Gamma_t)_{t \leq T}$ to a solution $(\bar{\Gamma}_t)_{t \leq T}$ to (\ref{Rinfty-CE}). Then, by Proposition \ref{Sp-pr. Rinfty prop} below we obtain a measure $\bar{\eta} \in \mathcal{P}(C_T\mathbb{R}^{\infty})$ with $\bar{\eta} \circ (e^{\infty}_t)^{-1} = \bar{\Gamma_t}$, which is concentrated on solution curves to (\ref{Rinfty-ODE}). Finally, we transfer $\bar{\eta}$ back to a measures $\eta \in \mathcal{P}(C_T\mathcal{SP})$ with the desired properties. Below, we denote by $\mathcal{F}C^1_b(\mathbb{R}^{\infty})$ the set of test functions of same type as in $\mathcal{F}C^2_b(\mathbb{R}^{\infty})$, but with $f \in C^1_b(\mathbb{R}^n)$ in place of $f \in \mathcal{F}C^2_b(\mathbb{R}^n)$.
 \begin{prop}\label{Sp-pr. Rinfty prop} [Superposition principle on $\mathbb{R}^{\infty}$, Thm. 7.1. \cite{ambrosio2014}]
 	Let $(\bar{\Gamma}_t)_{t \leq T}$ be a solution to (\ref{Rinfty-CE}) in the sense of Definition \ref{Def sol Rinfty-eq} (ii) with test functions $\mathcal{F}C^1_b(\mathbb{R}^{\infty})$ instead of $\mathcal{F}C^2_b(\mathbb{R}^{\infty})$. Then, there exists a Borel measures $\bar{\eta} \in \mathcal{P}(C_T\mathbb{R}^{\infty})$ concentrated on solutions to (\ref{Rinfty-ODE}) in the sense of Definition \ref{Def sol Rinfty-eq} (i) such that $$\bar{\eta} \circ (e^{\infty}_t)^{-1} = \bar{\Gamma}_t, \,\,t \leq T.$$
 \end{prop}
 We proceed to the proof of the main result.\\
 \\
 \textbf{Proof of Theorem \ref{main thm det case}:} Let $\Gamma = (\Gamma_t)_{t \leq T}$ be a weakly continuous solution to (\ref{P-CE}) as in Definition \ref{Def sol P-CE}. \\
 \textbf{Step 1: From (\ref{P-CE}) to (\ref{Rinfty-CE}):} Set 
 $$\bar{\Gamma}_t := \Gamma_t \circ G^{-1},$$
 with $G$ as in Lemma \ref{Lem aux G, J}, which corresponds to the fixed set of functions $\mathcal{G}$. Since $G$ is continuous, $(\bar{\Gamma}_t)_{t \leq T}$ is a weakly continuous curve of Borel subprobability measures on $\mathbb{R}^{\infty}$. We show that $(\bar{\Gamma}_t)_{t \leq T}$ solves (\ref{Rinfty-CE}). Indeed, the integrability condition (\ref{8}) is fulfilled, since $(\Gamma_t)_{t \leq T}$ fulfills Definition \ref{Def sol P-CE}. Further, since $\Gamma$ solves (\ref{P-CE}), we have for any $\mathcal{F}C^2_b(\mathcal{G}) \ni F: \mu \mapsto f\big(\mu(g_1),\dots,\mu(g_n)\big)$ and $t \in [0,T]$
 \begin{equation}\label{extra1}
 \int_0^t \int_{\mathcal{SP}}\mathbf{L}_sF(\mu)d\Gamma_s(\mu)ds = \int_{\mathcal{SP}}F(\mu)d\Gamma_t(\mu) - \int_{\mathcal{SP}}F(\mu)d\Gamma_0(\mu)
 \end{equation}	and hence, abbreviating $p_k \circ B(t,\cdot)$ by $B^k_t$ and setting $\bar{F} = f \circ \pi_n$ for $f$ as above, we have
 \begin{align*}
 \int_0^t \int_{\mathcal{SP}}\mathbf{L}_sF(\mu)d\Gamma_s(\mu)ds &= \int_0^t \int_{\mathcal{SP}}\sum_{k =1}^{n} (\partial_kf)\big(\mu(g_1),\dots,\mu(g_n)\big) \bigg(\int_{\mathbb{R}^d}\mathcal{L}_{s,\mu}g_k(x)d\mu(x)\bigg) \Gamma_s(\mu)ds \\& =
 \int_0^t \int_{\mathcal{SP}} \sum_{k=1}^n (\partial_k f)\big(\mu(g_1),\dots,\mu(g_n)\big)B^k_s(\mu)d\Gamma_s(\mu)ds \\&
 = \int_0^t \int_{\mathcal{SP}} \sum_{k=1}^n(\partial_k f)\big(p_1 \circ G(\mu),\dots,p_n\circ G(\mu)\big)\bar{B}^k_s \circ G(\mu)d\Gamma_s(\mu)ds\\&
 =\int_0^t \int_{\mathbb{R}^{\infty}} \nabla \bar{F}(z) \cdot \bar{B}_s(z)\bar{\Gamma}_s(z)ds
 \end{align*}and, furthermore, for each $s \in [0,T]$
 $$\int_{\mathcal{SP}}F(\mu)d\Gamma_s(\mu) = \int_{\mathcal{SP}}f\big(p_1 \circ G(\mu),\dots,p_n \circ G(\mu)\big)d\Gamma_s(\mu) = \int_{\mathbb{R}^{\infty}}\bar{F}(z)d\bar{\Gamma}_s(z).$$
 Comparing with (\ref{extra1}), it follows that $(\bar{\Gamma}_t)_{t \leq T}$ is a solution to (\ref{Rinfty-CE}) as claimed, because $F \in \mathcal{F}C^2_b(\mathcal{G})$ was arbitrary and hence $\bar{F}$ as above is arbitrary in $\mathcal{F}C^2_b(\mathbb{R}^{\infty})$. By standard approximation, one extends the above equation to test functions $\bar{F}$ from $\mathcal{F}C^1_b(\mathbb{R}^\infty)$.\\ 
 \\
 \textbf{Step 2: From (\ref{Rinfty-CE}) to (\ref{Rinfty-ODE}):} Proposition \ref{Sp-pr. Rinfty prop} implies the existence of a measure $\bar{\eta} \in \mathcal{P}(C_T\mathbb{R}^{\infty})$ such that
 \begin{enumerate}
 	\item [(i)] $\bar{\eta} \circ (e^{\infty}_t)^{-1} = \bar{\Gamma}_t$ for each $t \in [0,T]$
 	\item[(ii)] $\bar{\eta}$ is concentrated on solution paths of (\ref{Rinfty-ODE}).
 \end{enumerate}
\textbf{Step 3: From (\ref{Rinfty-ODE}) to (\ref{NLFPKE}):} We show that the measure $\eta := \bar{\eta} \circ (J^{-1})^{-1}$, with $J$ as in Lemma \ref{Lem aux G, J} fulfills all desired properties. Indeed, since
$$\bar{\eta} \circ (e^{\infty}_t)^{-1} = \bar{\Gamma}_t = \Gamma_t \circ G^{-1},$$
for each $t \in [0,T]$ we deduce that $\bar{\eta} \circ (e^{\infty}_t)^{-1}$ is concentrated on $G(\mathcal{SP})$. By Lemma \ref{Lem aux G, J}, $G(\mathcal{SP}) \subseteq \mathbb{R}^{\infty}$ is closed. Since by construction $\bar{\eta}$ is concentrated on continuous curves in $\mathbb{R}^{\infty}$, $\bar{\eta}$ is concentrated on $C_TG(\mathcal{SP})$. Further, $C_TG(\mathcal{SP}) \subseteq C_T\mathbb{R}^{\infty}$ is a measurable set and $J^{-1}: C_TG(\mathcal{SP}) \to C_T\mathcal{SP}$ is measurable by Lemma \ref{Lem aux G, J}. Therefore, we may define $\eta \in \mathcal{P}(C_T\mathcal{SP})$ via
$$ \eta := \bar{\eta} \circ (J^{-1})^{-1}.$$
It remains to verify $\eta \circ e_t^{-1} = \Gamma_t$ for all $t \in [0,T]$ and that $\eta$ is concentrated on subprobability solutions to (\ref{NLFPKE}). Concerning the first matter, we have
$$ \eta \circ e_t^{-1} = \bar{\eta} \circ (J^{-1})^{-1} \circ e_t^{-1} = \bar{\eta} \circ (e_t \circ J^{-1})^{-1}$$
and
$$\Gamma_t = \Gamma_t \circ (G^{-1} \circ G)^{-1} = \bar{\Gamma}_t \circ (G^{-1})^{-1} = \bar{\eta} \circ (G^{-1}\circ e_t^{\infty})^{-1}.$$
Since $e_t\circ J^{-1}$ and $G^{-1} \circ e_t^{\infty}$ coincide as measurable maps on $C_TG(\mathcal{SP})$ and it was shown above that $\bar{\eta}$ is concentrated on $C_TG(\mathcal{SP})$, we obtain
$$\eta \circ e_t^{-1} = \Gamma_t, \,\, t \leq T.$$

Concerning the second aspect, note that  by definition of $\eta$ and $\bar{\Gamma}_t$ and by the equality $e_t \circ J^{-1} = G^{-1}\circ e_t^\infty$,  \eqref{aux_revised1} for $\Gamma$ implies that $\eta$ is concentrated on vaguely continuous curves $t \mapsto \mu_t$ in $\mathcal{SP}$ with the global integrability property \eqref{2}  such that $t \mapsto G(\mu_t)$ is a solution to \eqref{Rinfty-ODE}. Each such curve $t \mapsto \mu_t$ is a subprobability solution to \eqref{NLFPKE}. Indeed, due to $\mu_t \in A_t$ $dt$-a.s., we have 
\begin{align*}
\frac{d}{dt}p_k \circ G(\mu_t) =&\, p_k \circ \bar{B}(t,G(\mu_t))\quad dt-a.s. \iff \frac{d}{dt}p_k \circ G(\mu_t) = p_k \circ B(t,\mu_t)\quad dt-a.s. \\&
\iff \frac{d}{dt}\int_{\mathbb{R}^d}g_k(x)d\mu_t(x) = \int_{\mathbb{R}^d}\mathcal{L}_{t,\mu_t}g_k(x)d\mu_t(x) \quad dt-a.s.\\&
\iff \int_{\mathbb{R}^d}g_k d\mu_t- \int_{\mathbb{R}^d}g_k d\mu_0 = \int_0^t \int_{\mathbb{R}^d}\mathcal{L}_{s,\mu_s}g_k (x)d\mu_s(x)ds, \,\, t \in [0,T],
\end{align*}
and Lemma \ref{Prop G early} (ii) applies.
It remains to prove the additional assertion about probability solutions. To this end, assume $\Gamma_0$ is concentrated on $\mathcal{P}$. Then, $\eta(e_0 \in \mathcal{P}) = 1$ and hence the claim follows by Remark \ref{Rem mass conserv}. \qed \\
\\
The final assertion of the theorem in particular implies: If $\Gamma_0 \in \mathcal{P}(\mathcal{P})$ for a weakly continuous solution $(\Gamma_t)_{t \leq T} \subseteq \mathcal{P}(\mathcal{SP})$ to (\ref{P-CE}), then $\Gamma_t \in \mathcal{P}(\mathcal{P})$ for each $t \leq T$. Of course, this is to be expected due to the global integrability condition in Definition \ref{Def sol P-CE}.
\begin{rem}\label{Rem explain why SP}
Finally, let us explain why we developed the above result for subprobability solutions to (\ref{NLFPKE}) although our principal interest is restricted to probability solutions. If we directly consider solution curves $(\Gamma_t)_{t \leq T}$ to (\ref{P-CE}) with $\Gamma_t \in \mathcal{P}(\mathcal{P})$, we cannot prove that $\eta$ in Theorem \ref{main thm det case} is concentrated on $C_T\mathcal{P}$ (in fact, not even $\eta(C_T\mathcal{P}) >0$ could be shown). Indeed, inspecting the proof above, one may only prove that $\eta \circ e_t^{-1}$ is concentrated on $\mathcal{P}$ for each $t\leq T$. But since $\mathcal{P} \subseteq \mathcal{SP}$ is not closed, curves in the support of $\eta$ may be proper subprobability-valued at single times. The deeper reason for this is that the range $G(\mathcal{P})$ of $G$ as in \ref{Prop G early} as a map on $\mathcal{P}$ with the weak topology is not closed in $\mathbb{R}^{\infty}$. It seems that one cannot resolve this issue by simply changing the function set $\mathcal{G}$, since there exists no countable set of functions, which allows for a characterization of weak instead of vague convergence as in Lemma \ref{Prop G early}.
Since $\mathcal{SP}$ with the vague topology is compact and the vague test function class $C_c(\mathbb{R}^d)$ is separable, it is feasible to carry out the entire development for subprobability measures as above.

We also mention that to our understanding there is no inherent reason why the superposition principle could not be extended to larger spaces of measures (e.g. spaces of signed measures), as long as its topology allows for a suitable identification with $\mathbb{R}^\infty$ as in our present case. Our principal motivation from a probabilistic viewpoint was to study curves of probability measures, and we were only forced to extend to $\mathcal{SP}$, the vague closure of $\mathcal{P}$, by the reasons outlined above. In order to replace $\mathcal{SP}$ by some larger space of measures $\mathcal{M}$, it seems indispensable that Lemma \ref{Lem aux G, J} remains true, i.e. that the range of $\mathcal{M}$ under a suitable homeomorphism  is closed in $\mathbb{R}^\infty$.

\end{rem}

\subsection{Consequences and applications}
The following existence- and uniqueness results immediately follow from the superposition principle Theorem \ref{main thm det case} and provide an equivalence between the nonlinear FPK-equation (\ref{NLFPKE}) and its linearized continuity equation (\ref{P-CE}). 
\begin{kor}
	Let $\mu_0 \in \mathcal{SP}$ and assume there exists a solution to (\ref{P-CE}) with initial condition $\delta_{\mu_0}$. Then, there exists a subprobability solution to (\ref{NLFPKE}) with initial condition $\mu_0$. Moreover, if $\mu_0\in \mathcal{P}$, then there exists a probability solution to (\ref{NLFPKE}) with initial condition $\mu_0$.
\end{kor}
\begin{proof}
	By Theorem \ref{main thm det case} there exists a  probability measure $\eta$ concentrated on subprobability solutions to (\ref{NLFPKE}) with $\eta \circ e_0^{-1} = \delta_{\mu_0}$. Hence, at least one such solution to (\ref{NLFPKE}) with initial condition $\mu_0$ exists. The second assertion is treated similarly.
\end{proof}
\begin{kor}\label{Cor uniqueness det}
	Let $\mu_0 \in \mathcal{SP}$ and assume there exists at most one vaguely continuous subprobability solution to (\ref{NLFPKE}) with initial condition $\mu_0$. Then, there exists also at most one weakly continuous solution $(\Gamma_t)_{t \leq T}$ to (\ref{P-CE}) with initial condition $\delta_{\mu_0}$. If $\mu_0 \in \mathcal{P}$, then, in the case of existence, $\Gamma_t(\mathcal{P}) = 1$ for each $t \in [0,T]$.
\end{kor}
\begin{proof}
	Let $\Gamma^{(1)}$ and $\Gamma^{(2)}$ be weakly continuous solutions to (\ref{P-CE}) with $\Gamma^{(i)}_0 = \delta_{\mu_0}$ for $i \in \{1,2\}$. By Theorem \ref{main thm det case}, there exist probability measures $\eta^{(i)}$, $i \in \{1,2\}$, concentrated on subprobability solutions to (\ref{NLFPKE}) with initial condition $\mu_0$ such that $\eta^{(i)} \circ e_t^{-1} = \Gamma_t^{(i)}$ for each $t \in [0,T]$ and $i \in \{1,2\}$. By assumption, we obtain $\eta^{(1)}= \delta_{\mu}= \eta^{(2)}$ for a unique element $\mu \in C_T\mathcal{SP}$ and thus also $\Gamma^{(1)} = \Gamma^{(2)}$. If $\mu_0 \in \mathcal{P}$, then $\mu \in C_T\mathcal{P}$ by Remark \ref{Rem mass conserv}, which gives the second assertion. 
\end{proof}
\subsubsection{Application to coupled nonlinear-linear Fokker-Planck-Kolmogorov equations}
Using the superposition principle, we prove an open conjecture posed in \cite{Rckner-lin.-paper}. Let us shortly recapitulate the necessary framework. In \cite{Rckner-lin.-paper}, the authors consider a coupled nonlinear-linear FPK-equation of type 
\begin{equation}\label{9}
\begin{cases}
\partial_t \mu_t = \mathcal{L}^*_{t,\mu_t}\mu_t \\
\partial_t \nu_t = \mathcal{L}^*_{t,\mu_t}\nu_t,
\end{cases}
\end{equation}
i.e. comparing to our situation the first nonlinear equation is of type (\ref{NLFPKE}) and the second (linear) equation is obtained by "freezing" a solution $(\mu_t)_{t \leq T}$ to the first equation in the nonlinearity spot of $\mathcal{L}$. For an initial condition $(\bar{\mu},\bar{\nu}) \in \mathcal{P}\times \mathcal{P}$, (\ref{9}) is said to have a \textit{unique solution}, if there exists a unique probability solution $(\mu_t)_{t \leq T}$ to the first equation in the sense of Definition \ref{Sol NLFPKE} with $\mu_0 = \bar{\mu}$ and a unique weakly continuous curve $(\nu_t)_{t \leq T} \subseteq \mathcal{P}$, which solves the second equation with fixed coefficient $\mu_t$ with $\nu_0 = \bar{\nu}$ (we refer to \cite{Rckner-lin.-paper} for more details). The authors associate a linear continuity equation on $\mathbb{R}^d \times \mathcal{P}$ to (\ref{9}) in the following sense: Let $\mathbb{L}$ be the operator acting on functions 
$$\mathcal{C} := \big\{\Phi: (x,\mu) \mapsto \varphi(x)F(\mu)  | \varphi \in C^2_c(\mathbb{R}^d), F \in \mathcal{F}C^2_b(\mathcal{P}) \big\},$$
via 
$$\mathbb{L}_t\Phi(x,\mu) := \mathcal{L}_{t,\mu}\Phi(\cdot, \mu)(x)+\mathbf{L}_t\Phi(x, \cdot)(\mu),$$with $\mathcal{L}$ as in (\ref{1}) and $\mathbf{L}$ as in (\ref{5}).
Consider the continuity equation 
\begin{equation}\label{10}
\partial_t \Lambda_t = \mathbb{L}_t^*\Lambda_t, \,\, t \in [0,T]
\end{equation}for weakly continuous curves of Borel probability measures on $\mathbb{R}^d \times \mathcal{P}$. The exact notion of solution can be found in \cite{Rckner-lin.-paper}, where also the following observation is made: A pair $(\mu_t,\nu_t)_{t \leq T}$ solves (\ref{9}) if and only if $\Lambda_t := \nu_t \times \delta_{\mu_t}$ solves (\ref{10}). Using our main result, we prove the following conjecture posed in Remark 4.4. of \cite{Rckner-lin.-paper}.
\begin{prop}\label{Prop conj Rckner}
	If $(\mu_t,\nu_t)_{t \leq T}$ is the unique solution to (\ref{9}) with initial condition $(\bar{\mu},\bar{\nu}) \in \mathcal{P}\times \mathcal{P}$, then $(\nu_t \times \delta_{\mu_t})_{t \leq T}$ is the unique solution to (\ref{10}) with initial condition $\bar{\nu}\times \delta_{\bar{\mu}}$.
\end{prop}
\begin{proof}
	By Corollary (\ref{Cor uniqueness det}), the unique solution to (\ref{P-CE}) with initial condition $\delta_{\bar{\mu}}$ is $(\delta_{\mu_t})_{t \leq T}$. Let $(\Lambda^{(1)}_t)_{t \leq T}$ and $(\Lambda^{(2)}_t)_{t \leq T}$ be two solutions to (\ref{10}) with initial condition $\bar{\nu}\times \delta_{\bar{\mu}}$. It is straightforward to check that the curves of second marginals $(\Lambda_t^{(1)}\circ \varPi_2^{-1})_{t \leq T}$ and $(\Lambda_t^{(2)}\circ \varPi_2^{-1})_{t \leq T}$ are probability solutions to (\ref{P-CE}) with initial condition $\delta_{\bar{\mu}}$ (where we denote by $\varPi_2$ the projection from $\mathbb{R}^d \times \mathcal{P}$ onto the second coordinate). Hence, for each $t \in [0,T]$
	$$\Lambda_t^{(1)} \circ \varPi_2^{-1} = \delta_{\mu_t} = \Lambda^{(2)}_t \circ \varPi_2^{-1}.$$
	Consequently, $\Lambda_t^{(i)}$ is of product type, i.e. $\Lambda^{(i)}_t = \gamma_t^{(i)}\times \delta_{\mu_t}$ for weakly continuous curves $(\gamma^{(i)}_t)_{t \leq T} \subseteq \mathcal{P}$, $i \in \{1,2\}$. It is immediate to show that each curve $\gamma^{(i)}$ solves the second equation of (\ref{9}) with fixed $\mu_t$ and initial condition $\bar{\nu}$. Hence, $\gamma^{(i)}_t = \nu_t$ for each $t \in [0,T]$ and $i \in \{1,2\}$, which implies $\Lambda^{(1)}_t = \Lambda^{(2)}_t$. Hence, the unique solution to (\ref{10}) with initial condition $\bar{\nu}\times \delta_{\bar{\mu}}$ is given by $(\nu_t \times \delta_{\mu_t})_{t \leq T}$. 
\end{proof}
\section{Superposition Principle for stochastic nonlinear Fokker-Planck-Kolmogorov Equations}
We make use of the following notation specific to the stochastic case.\\
\\
For two real-valued $n\times n$ matrices $A,B$ we write $A$:$B = \sum_{k,l = 1}^n A_{kl}B_{kl}$. We use the same notation for $A = (A_{kl})_{k,l \geq 1}$ and $B = (B_{kl})_{k,l \geq 1}$, if either $A$ or $B$ contain only finitely many non-trivial entries.\\
For the Hilbert space $\ell^2$ with topology induced by the usual inner product $\langle\cdot,\cdot \rangle _{\ell^2}$ and norm $||\cdot||_{\ell^2}$, we denote the space of continuous $\ell^2$-valued functions on $[0,T]$ by $C_T\ell^2$. On $\ell^2$ and $C_T\ell^2$, we unambiguously use the same notation $e_t, p_i$ and $\pi_n$ as on $\mathbb{R}^{\infty}$ and $C_T\mathbb{R}^{\infty}$ in the previous section. Reminiscent to the previous section, we set $\mathcal{B}(C_T\ell^2) = \sigma(e_t, t \in [0,T])$ and denote the set of probability measures on this space by $\mathcal{P}(C_T\ell^2)$. For $\sigma$-algebras $\mathcal{A}_1$, $\mathcal{A}_2$, we denote by $\mathcal{A}_1 \bigvee \mathcal{A}_2$ the \textit{$\sigma$-algebra generated by $\mathcal{A}_1$ and $\mathcal{A}_2$}.\\
\\
We call a filtered probability space $(\Omega, \mathcal{F}, (\mathcal{F}_t)_{t \leq T},\mathbb{P})$ \textit{complete}, provided both $\mathcal{F}$ and $\mathcal{F}_0$ contain all subsets of $\mathbb{P}$-negligible sets $N \in \mathcal{F}$ (i.e. $\mathbb{P}(N) = 0$). This notion does not require $(\mathcal{F}_t)_{t \leq T}$ to be right-continuous. A real-valued Wiener process $W = (W_t)_{t \leq T}$ on such a probability space is called an $\mathcal{F}_t$-\textit{Wiener process}, if $W_t$ is $\mathcal{F}_t$-adapted and $W_u-W_t$ is independent of $\mathcal{F}_t$ for each $0 \leq t \leq u \leq T$. Pathwise properties of stochastic processes such as continuity are to be understood up to a negligible set with respect to the underlying measure.\\
\\
As in the previous section, we consider $\mathcal{SP}$ as a compact Polish space with the vague topology. Let $d_1 \geq1$ and consider product-measurable coefficients on $[0,T]\times \mathcal{SP}\times \mathbb{R}^d$
$$a(t,\mu,x)=(a_{ij}(t,\mu,x)) \in \mathbb{S}^+_d,\,\,b(t,\mu,x) = (b_i(t,\mu,x))_{i \leq d} \in \mathbb{R}^d,\, \sigma(t,\mu,x)= (\sigma_{ij}(t,\mu,x))_{i,j \leq d} \in \mathbb{R}^{d\times d_1}$$
such that $\sigma$ is bounded, 
and let $\mathcal{L}$ be as before, i.e.
$$\mathcal{L}_{t, \mu} \varphi(x) = b_i(t,\mu,x)\partial_i \varphi(x)+ a_{ij}(t,\mu,x)\partial^2_{ij}\varphi(x)$$
for $\varphi \in C^2(\mathbb{R}^d)$ and $(t,\mu,x) \in [0,T]\times \mathcal{SP}\times \mathbb{R}^d$.\\ 
\\
In contrast to the deterministic framework of the previous section, here we consider nonlinear \textit{stochastic} FPK-equations of type (\ref{SNLFPKE}) on $[0,T]$, to be understood in distributional sense as follows. With slight abuse of notation, for $\sigma \in \mathbb{R}^{d\times d_1}$ and $x \in \mathbb{R}^d$, we write $\sigma \cdot x = (\sum_{i=1}^d\sigma^{ik}x_i)_{k \leq d_1}$, which is consistent with the standard inner product notation $\sigma \cdot x$ in the case $d_1 =1$. 
\begin{dfn}\label{Def sol SNLFPKE}
	\begin{enumerate}
		\item [(i)] A pair $(\mu, W)$ consisting of an $\mathcal{F}_t$-adapted vaguely continuous $\mathcal{SP}$-valued stochastic process $\mu = (\mu_t)_{t \leq T}$ and an $\mathcal{F}_t$-adapted, $d_1$-dimensional Wiener process $W = (W_t)_{t \leq T}$ on a complete probability space $(\Omega, \mathcal{F}, (\mathcal{F}_t)_{t \leq T},\mathbb{P})$ is a \textit{subprobability solution to }(\ref{SNLFPKE}), provided 
		\begin{equation}\label{2_int_SNLFPKE}
			\int_0^T\int_{\mathbb{R}^d}|b_i(t,\mu_t,x)|+|a_{ij}(t,\mu_t,x)| + |\sigma_{ik}(t,\mu_t,x)|^2d\mu_t(x)dt < \infty \quad \mathbb{P}\text{-a.s.}
		\end{equation}
		for each $i,j \leq d, k \leq d_1$, and 
		\begin{equation}\label{11}
		\int_{\mathbb{R}^d} \varphi(x) d\mu_t(x) - 	\int_{\mathbb{R}^d} \varphi(x) d\mu_0(x) = \int_0^t \int_{\mathbb{R}^d}\mathcal{L}_{s,\mu_s}\varphi(x) d\mu_s(x)ds + \int_0^t \int_{\mathbb{R}^d} \sigma(s,\mu_s,x)\cdot \nabla \varphi(x)d \mu_s(x)dW_s
		\end{equation}holds $\mathbb{P}$-a.s. for each $t \in [0,T]$ and $\varphi \in C^2_c(\mathbb{R}^d)$.
		\item[(ii)] A \textit{probability solution to } (\ref{SNLFPKE}) is a pair as above such that $\mu$ is a $\mathcal{P}$-valued process $(\mu_t)_{t \leq T}$ with weakly continuous paths.
	\end{enumerate}
\end{dfn}
\begin{rem}
	\begin{enumerate}
		\item [(i)] Since $C^2_c(\mathbb{R}^d)$ is separable with respect to uniform convergence and since the paths $t \mapsto \mu_t(\omega)$ are vaguely continuous, the exceptional sets in the above definition can be chosen independently of $\varphi $ and $t$.
		\item[(ii)] The first integral on the right-hand side of (\ref{11}) is a pathwise (that is, for individual fixed $\omega \in \Omega$) integral with respect to the finite measure $\mu_s(\omega)ds$ on $[0,T]\times \mathbb{R}^d$. The second integral is a stochastic integral, which is defined, since the integrand
		$$(t,\omega) \mapsto \int_{\mathbb{R}^d}\sigma(t,\mu_t(\omega),x)\cdot \nabla \varphi(x)d\mu_t(\omega)(x)$$
		is $\mathbb{R}^{d_1}$-valued, bounded, product-measurable and $\mathcal{F}_t$-adapted (Thm. 3.8 \cite{ChungWilliamsStochInt}). More precisely,
		$$\int_0^t \int_{\mathbb{R}^d} \sigma(s,\mu_s,x)\cdot \nabla \varphi(x)d \mu_s(x)dW_s = \sum_{\alpha=1}^{d_1} \int_0^t\int_{\mathbb{R}^d}\sigma^{\alpha} \cdot \nabla \varphi d\mu_sdW^{\alpha}_s,$$where $\sigma^{\alpha}= (\sigma^{i\alpha})_{i \leq d}$ denotes the $\alpha$-th column of $\sigma$ and the components $W^{\alpha}$, $\alpha \leq d_1$, of $W$ are real, independent Wiener processes.
	\end{enumerate}
\end{rem}
By the global integrability assumption \eqref{2_int_SNLFPKE} and since $\sigma$ is bounded, we obtain (in analogy to Remark \ref{Rem mass conserv}) the following conservation of mass, which we use to prove the final assertion of the main result Theorem \ref{main thm stoch case}.
\begin{lem}\label{Lem consv of mass stochastic}
	Let $(\mu_t)_{t \leq T}$ be a subprobability solution to (\ref{SNLFPKE}). If $\mu_0 \in \mathcal{P}$ $\mathbb{P}$-a.s., then the paths of $t \mapsto \mu_t$ are $\mathcal{P}$-valued $\mathbb{P}$-a.s. and, hence, in particular weakly continuous.
\end{lem}
\begin{proof}
	Let $(\varphi_k)_{k \geq 1} \subseteq C^2_c(\mathbb{R}^d)$ approximate the constant function $1$ as in Remark \ref{Rem mass conserv}. Then, by Itô-isometry, for each $t \in [0,T]$, there exists a subsequence $(k^t_l)_{l \geq 1} = (k_l)_{l \geq 1}$ such that
	\begin{equation}\label{12}
	\int_0^t\int_{\mathbb{R}^d}\sigma(s,\mu_s, x)\cdot \nabla \varphi_{k_l}(x)d\mu_s(x)dW_s \underset{l \to \infty}{\longrightarrow}0 \,\, \mathbb{P}\text{-a.s.}
	\end{equation}
	Since the stochastic integral is continuous in $t$, a classical diagonal argument yields that there exists a subsequence $(k_l)_{l \geq 1}$ along which (\ref{12}) holds for all $t \in [0,T]$ on a set of full $\mathbb{P}$-measure, independent of $t$. Let $\omega' \in \Omega$ be from this set such that also $\mu_0(\omega') \in \mathcal{P}$ and (\ref{11}) holds for each $t$ and $\varphi$. Note that the set of all such $\omega'$ has full $\mathbb{P}$-measure. Then, similar to the reasoning in Remark \ref{Rem mass conserv} and by using (\ref{12}), considering (\ref{11}) for such $\omega'$ with $\varphi_{k_l}$ in place of $\varphi$ for the limit $l \longrightarrow +\infty$, we obtain
	$$\mu_t(\omega')(\mathbb{R}^d) = \mu_0(\omega')(\mathbb{R}^d), \,\, t\in [0,T]$$
	and hence the result.
\end{proof}
Note that the above proof can be adjusted to extend (\ref{11}) to each $\varphi \in C^2_b(\mathbb{R}^d)$. 
\subsubsection*{Embedding $\mathcal{SP}$ into $\ell^2$}
In comparison with the deterministic case, we still consider $\mathcal{SP}$ as a manifold-like space with tangent spaces $T_{\mu}\mathcal{SP} = L^2(\mathbb{R}^d,\mathbb{R}^d;\mu)$ as before. However, instead of embedding into $\mathbb{R}^{\infty}$ by $G$ as in the previous section, now we need a global chart 
$$H: \mathcal{SP} \to \ell^2$$
in order to handle the stochastic integral term later on. To this end, we replace the set of functions $\mathcal{G} = \{g_i, i \geq 1\}$ of the deterministic case by
\begin{equation}\label{13}
\mathcal{H} := \{h_i\}_{i \geq 1}, \, h_i := 2^{-i}\frac{g_i}{||g_i||_{C^2_b}}
\end{equation}and consider the map
$$H : \mathcal{SP} \to \ell^2,\,\, H: \mu \mapsto (\mu(h_i))_{i \geq 1}.$$
The following lemma collects useful properties of $\mathcal{H}$ and $H$, which are in the spirit of Lemma \ref{Prop G early} and \ref{Lem aux G, J}. We point out that we could have used the function class $\mathcal{H}$ instead of $\mathcal{G}$ already in Section 3, but we decided to pass from $\mathcal{G}$ to $\mathcal{H}$ at this point in order to stress the technical adjustments necessary due to the stochastic case.
\begin{lem}\label{Lem aux H}
	\begin{enumerate}
		\item [(i)] The set $\mathcal{H}$ is measure-determining. Further, a process $(\mu_t)_{t \leq T}$ as in Definition \ref{Def sol SNLFPKE} is a solution to (\ref{SNLFPKE}) if and only if (\ref{11}) holds for each $h_i \in \mathcal{H}$ in place of $\varphi$.
		\item[(ii)] $H$ is a homeomorphism between $\mathcal{SP}$ and its range $H(\mathcal{SP}) \subseteq \ell^2$, endowed with the $\ell^2$-subspace topology. In particular, $H(\mathcal{SP}) \subseteq \ell^2$ is compact.
	
	\end{enumerate}
\end{lem}
\begin{proof}
	\begin{enumerate}
		\item [(i)] The first claim is obvious, since $\mathcal{G}$ is measure-determining. Concerning the second claim, note that it is clearly sufficient to have (\ref{11}) for each $\varphi \in C_c^2(\mathbb{R}^d)$ with $||\varphi||_{C^2_b} \leq 1$. Since the functions $||g_i||^{-1}_{C^2_b}g_i$ are dense in the unit ball of $C^2_c(\mathbb{R}^d)$ with respect to $||\cdot||_{C^2_b}$, it is sufficient to have (\ref{11}) for each such normalized function. Indeed, if $\varphi_k \underset{k \to \infty}{\longrightarrow}\varphi$ uniformly up to second-order partial derivatives, then by Itô-isometry
		\begin{equation*}
		\mathbb{E}\bigg[\bigg(\int_0^t\int_{\mathbb{R}^d}\sigma(s,\mu_s,\cdot)\cdot \nabla (\varphi_k-\varphi)d\mu_sdW_s\bigg)^2\bigg] = \mathbb{E}\bigg[\int_0^t\bigg(\int_{\mathbb{R}^d}\sigma(s,\mu_s,\cdot)\cdot\nabla(\varphi_k-\varphi)d\mu_s \bigg)^2ds\bigg],
		\end{equation*}which converges to $0$ as $k \longrightarrow \infty$ due to the boundedness of $\sigma$. Hence, along a subsequence $(k_l)_{l \geq 1}$, we have a.s.
		\begin{equation*}
		\int_0^t \int_{\mathbb{R}^d}\sigma(s,\mu_s,x) \cdot \nabla \varphi_{k_l}(x)d\mu_s(x)dW_s \underset{l \to \infty}{\longrightarrow} \int_0^t \int_{\mathbb{R}^d}\sigma(s,\mu_s,x)\cdot \nabla\varphi(x)d\mu_s(x)dW_s.
		\end{equation*}The a.s.-convergence of all other terms in (\ref{11}) is clear. Therefore, it is sufficient to require (\ref{11}) for a dense subset of the unit ball of $C^2_c(\mathbb{R}^d)$. Clearly, this yields at once that it is sufficient to have (\ref{11}) for each $h_i \in \mathcal{H}$. 
		\item[(ii)] By definition, $H$ maps into $\ell^2$. Since $\mathcal{H}$ is measure-determining, $H$ is one-to-one, hence bijective onto its range. If $\mu_n \underset{n \to \infty}{\longrightarrow} \mu$ vaguely in $\mathcal{SP}$, clearly $H(\mu_n)$ converges to $H(\mu)$ in the product topology. Since for any $i \geq 1$
		$$\underset{n \geq 1}{\text{sup}}|H(\mu_n)_i| \leq 2^{-i},$$
		the convergence holds in $\ell^2$ as well, which implies continuity of $H$. In particular, $H(\mathcal{SP}) \subseteq \ell^2$ is compact. Conversely, if $H(\mu_n)$ converges in $\ell^2$ to some $z = (z_i)_{i \geq 1}$, then, by closedness of $H(\mathcal{SP}) \subseteq \ell^2$, we have $z=H(\mu)$ for a unique element $\mu \in \mathcal{SP}$ and $\mu_n \underset{n \to \infty}{\longrightarrow}\mu$ vaguely. Indeed, the latter follows as in Lemma \ref{Lem aux G, J} (i). 	
	\end{enumerate} 
\end{proof}
For consistency of notation, below we denote the test function class of the manifold-like space $\mathcal{SP}$ by $\mathcal{F}C^2_b(\mathcal{H})$ to stress that the base functions $g_i$ are now replaced by $h_i \in \mathcal{H}$. However, the class of test functions remains unchanged, because the transition from $g_i$ to $h_i$ can be incorporated in the choice of $f$.
\subsubsection*{Linearization of (\ref{SNLFPKE})}
As in the deterministic case, also for the stochastic nonlinear equation (\ref{SNLFPKE}) one can consider an associated linear equation for curves in $\mathcal{P}(\mathcal{SP})$. To the best of our knowledge, such a linearization for stochastic FPK-equations has not yet been considered in the literature. Of course, the basic idea stems from the deterministic case \cite{Rckner-lin.-paper} discussed in the previous section. From Itô's formula one expects this linearized equation to be of second-order.\\
\\
Let $\big((\mu_t)_{t \leq T},W\big)$ be a subprobability solution to (\ref{SNLFPKE}) (with underlying measure $\mathbb{P}$) and choose any $F: \mu \mapsto f\big(\mu(h_1),\dots,\mu(h_n)\big)$ from $\mathcal{F}C^2_b(\mathcal{H})$. Again, we abbreviate $b(t,\mu) := b(t,\mu,\cdot)$ and similarly for $a$ and $\sigma$. By Itô's formula, we have $\mathbb{P}$-a.s.
\begin{align*}
F(\mu_t)-F(\mu_0) &= \int_0^t \big\langle\nabla^{\mathcal{SP}}F(\mu), b(s,\mu)+a(s,\mu)\nabla \big \rangle_{L^2(\mu_s)}ds \\&
+\frac{1}{2}\sum_{\alpha=1}^{d_1}\int_0^t \sum_{k,l=1}^n (\partial_{kl}f)(\mu(h_1),\dots,\mu(h_n))\bigg(\int_{\mathbb{R}^d}\sigma^{\alpha}(s,\mu)\cdot \nabla h_k d\mu\bigg)\bigg(\int_{\mathbb{R}^d}\sigma^{\alpha}(s,\mu)\cdot \nabla h_l d\mu\bigg)ds \\&+ M_t^F
,
\end{align*}with the martingale $M^F$ given as
$$M_t^F:= \sum_{\alpha=1}^{d_1}\int_0^t\bigg[ \sum_{l=1}^{n}(\partial_lf)\big(\mu(h_1),\dots,\mu(h_n)\big) \int_{\mathbb{R}^d}\sigma^{\alpha} \cdot \nabla h_l d\mu_s \bigg]dW_s^{\alpha}. $$Since $M^F_0 = 0$ $\mathbb{P}$-a.s., integrating with respect to $\mathbb{P}$ and defining the curve of measures in $\mathbb{P}(\mathcal{SP})$
$$\Gamma_t := \mathbb{P}\circ \mu_t^{-1}, \,\, t\leq T$$yields
\begin{align}\label{prelim eq}
\notag &\int_{\mathcal{SP}}F(\mu)d\Gamma_t(\mu) - \int_{\mathcal{SP}}F(\mu)d\Gamma_0(\mu)= \int_0^t\int_{\mathcal{SP}} \big\langle\nabla^{\mathcal{SP}}F(\mu), b(s,\mu)+a(s,\mu)\nabla \big \rangle_{L^2(\mu)}d\Gamma_s(\mu)ds \\& 
	+\frac{1}{2}\sum_{\alpha=1}^{d_1}\int_0^t \int_{\mathcal{SP}}\sum_{k,l=1}^n (\partial_{kl}f)\big(\mu(h_1),\dots,\mu(h_n)\big)\bigg(\int_{\mathbb{R}^d}\sigma^{\alpha}(s,\mu)\cdot \nabla h_k d\mu\bigg)\bigg(\int_{\mathbb{R}^d}\sigma^{\alpha}(s,\mu)\cdot \nabla h_l d\mu\bigg)d\Gamma_s(\mu)ds.
\end{align}As for the first-order term, which is interpreted as the pairing of the gradient $\nabla^{\mathcal{SP}}F$ with the inhomogeneous vector field $b+a\nabla$ in the tangent bundle $T\mathcal{SP}$, also the second-order term allows for a geometric interpretation: Recall that for a smooth, real function $F$ on a Riemannian manifold $M$ with tangent bundle $TM$, the Hessian $Hess(F)_p$ at $p \in M$ is a bilinear form on $T_pM$ with
\begin{equation}\label{Hess gen mf}
Hess(F)_p(\eta_p, \xi_p) = \big\langle \nabla^L_{\eta_p}\nabla F(p), \xi_p \big \rangle_{T_pM},\,\, \eta_p, \xi_p \in T_pM, 
\end{equation}where $\nabla^L: TM \times TM \to TM$ denotes the Levi-Civita-connection on $M$, the unique affine connection compatible with the metric tensor on $M$ and $\nabla$ denotes the usual gradient on $M$. Intuitively, $\nabla^L_{\eta_p}\nabla F(p)$ denotes the change of the vector field $\nabla F$ in direction $\eta_p$ at $p$. Recall that we consider $\mathcal{SP}$ as a manifold-like space with gradient $\nabla^{\mathcal{SP}}$ and that hence the reasonable notion of the Levi-Civita connection$\nabla^{L,\mathcal{SP}}$ on $\mathcal{SP}$ for $\sigma \in T_{\mu}\mathcal{SP} = L^2(\mathbb{R}^d,\mathbb{R}^d;\mu), Y \in T\mathcal{SP}$ at $\mu$ is given by
$$\nabla ^{L,\mathcal{SP}}_{\sigma}Y(\mu) = \big \langle \nabla^{\mathcal{SP}}Y(\mu), \sigma\big \rangle_{T_{\mu}\mathcal{SP}},$$whenver $\nabla^{\mathcal{SP}}Y$ is defined in $T\mathcal{SP}$. For the representation of $Hess(F)$ for a test function $F \in \mathcal{F}C^2_b(\mathcal{H})$, we need to set $Y = \nabla^{\mathcal{SP}}F$. In this case, we can indeed make sense of
$$ (\nabla^{\mathcal{SP}})^2F:= \nabla^{\mathcal{SP}}\nabla^{\mathcal{SP}}F,$$ because the gradient
$$\mu \mapsto \nabla^{\mathcal{SP}}F(\mu) = \sum_{k=1}^{n}(\partial_kf)\big(\mu(h_1),\dots,\mu(h_n)\big)\nabla h_k$$ is a linear combinations of the "$\mathcal{F}C^2_b(\mathcal{H})$-like" functions $\mu \mapsto \partial_kf\big(\mu(h_1),\dots,\mu(h_n)\big)$. The linear combination has to be understood in an $x$-wise sense with coefficient functions $\nabla h_k$, which are independent of the variable of interest $\mu$. Denoting $F_k(\mu):= (\partial_kf)\big(\mu(h_1),\dots,\mu(h_n)\big)$, we then define
\begin{equation}\label{Def nabla squared}
(\nabla^{\mathcal{SP}})^2F(\mu) (x,y) := \sum_{k=1}^n\big(\nabla^{\mathcal{SP}}F_k(\mu)\big)(y)\nabla h_k(x),\,\, (x,y) \in \mathbb{R}^d\times \mathbb{R}^d. 
\end{equation}
Consequently, we have a reasonable notion of the Levi-Civita connection on $\mathcal{SP}$ at $\mu$ for $\sigma \in T_{\mu}\mathcal{SP}$ and $\nabla^{\mathcal{SP}}F$ for $F \in \mathcal{F}C^2_b(\mathcal{H})$ as
\begin{equation}
\nabla^{L,\mathcal{SP}}_{\sigma}\nabla^{\mathcal{SP}}F(\mu)  := \big \langle (\nabla^{\mathcal{SP}})^2F(\mu), \sigma \big \rangle_{T_{\mu}\mathcal{SP}} = \sum_{k,l=1}^{n}(\partial_{kl}f)\big(\mu(h_1),\dots,\mu(h_n)\big)\nabla h_k\bigg(\int_{\mathbb{R}^d}\sigma \cdot \nabla h_l d\mu\bigg).
\end{equation}The section $(\nabla^{\mathcal{SP}})^2F$ in $T\mathcal{SP}^*\otimes T\mathcal{SP}^*$ (and hence $\nabla^{L,\mathcal{SP}}_{\sigma}\nabla^{\mathcal{SP}}F$ and $Hess(F)$ below) is independent of the particular representation of $F$ in (\ref{Def nabla squared}). Indeed, we have (c.f. Appendix A \cite{Rckner-lin.-paper}) for $$\gamma^{\sigma}_{\mu}(t):= \mu \circ (\text{Id}+t \sigma)^{-1}$$ the following pointwise (in $x \in \mathbb{R}^d$) equality for each $\mu \in \mathcal{SP}, \sigma \in L^2(\mathbb{R}^d,\mathbb{R}^d;\mu)$
\begin{align*}
\frac{d}{dt}\nabla^{\mathcal{SP}}F\big(\gamma^{\sigma}_{\mu}(t)\big) &= \sum_{k=1}^n\bigg[\frac{d}{dt}(\partial_kf)\big(\gamma^{\sigma}_{\mu}(t)(h_1),\dots,\gamma^{\sigma}_{\mu}(t)(h_n)\big)\bigg]\nabla h_k \\& = \sum_{k,l=1}^n(\partial_{kl}f)\big(\mu(h_1),\dots,\mu(h_n)\big)\big \langle \nabla h_l, \sigma \big \rangle_{L^2(\mu)} \nabla h_k \\& = 
\big \langle (\nabla^{\mathcal{SP}})^2F(\mu),\sigma \big \rangle_{L^2(\mu)}.
\end{align*}Since the gradient $\nabla^{\mathcal{SP}}F$ is independent of the particular representation of $F$ and $\sigma \in L^2(\mathbb{R}^d,\mathbb{R}^d;\mu)$ is arbitrary, also $(\nabla^{\mathcal{SP}})^2F$ is independent of the representation of $F$. \\

Considering (\ref{Hess gen mf}), we then set for $F \in \mathcal{F}C^2_b(\mathcal{H})$ and $\sigma, \tilde{\sigma} \in L^2(\mathbb{R}^d,\mathbb{R}^d;\mu)$
\begin{equation}
Hess(F)(\mu): (\sigma, \tilde{\sigma}) \mapsto \sum_{k,l=1}^{n}(\partial_{kl}f)\big(\mu(h_1),\dots,\mu(h_n)\big)\bigg(\int_{\mathbb{R}^d}\sigma \cdot \nabla h_l d\mu\bigg)\bigg(\int_{\mathbb{R}^d}\tilde{\sigma} \cdot \nabla h_k d\mu\bigg),
\end{equation}which is a (symmetric) bilinear form on $T_{\mu}\mathcal{SP}$ and rewrite (\ref{prelim eq}) as
\begin{align}\label{P-FPKE eq}
\int_{\mathcal{SP}}Fd\Gamma_t - \int_{\mathcal{SP}}Fd\Gamma_0= \int_0^t\int_{\mathcal{SP}} \big\langle\nabla^{\mathcal{SP}}F, b_s+a_s\nabla \big \rangle_{L^2} + \frac{1}{2}\sum_{\alpha=1}^{d_1} Hess(F)(\sigma_s^{\alpha},\sigma_s^{\alpha})d\Gamma_sds
\end{align}(with $b_s: (\mu,x) \mapsto b(s,\mu,x)$ and similarly for $a_s$ and $\sigma_s$). Introducing the second-order operator $\mathbf{L}^{(2)}$, acting on $F \in \mathcal{F}C^2_b(\mathcal{H})$ via
$$\mathbf{L}^{(2)}_tF(\mu) = \big\langle\nabla^{\mathcal{SP}}F, b(t,\mu)+a(t,\mu)\nabla \big \rangle_{L^2(\mu)} + \frac{1}{2}\sum_{\alpha=1}^{d_1} Hess(F)\big(\sigma^{\alpha}(t,\mu),\sigma^{\alpha}(t,\mu)\big),$$
we arrive at the distributional formulation of (\ref{P-FPKE})
$$\partial_t\Gamma_t = (\mathbf{L}_t^{(2)})^*\Gamma_t, \,\, t \leq T,$$as in the introduction.
\begin{rem}
	Equation (\ref{P-CE}) is the natural analogue to second-order FPK-equations over Euclidean spaces. Indeed, for a stochastic equation on $\mathbb{R}^d$
		\begin{equation}\label{StochEq}
		dX_t = b(t,X_t)dt+\sigma(t,X_t)dW_t,
		\end{equation}
	 by Itô's formula, the corresponding linear second-order equation for measures in distributional form is 
	$$\partial_t\mu_t = \big(\mathcal{L}^{(2)}_t\big)^*\mu_t$$
	with $$\mathcal{L}^{(2)}_tf = \nabla f\cdot b_t+\frac{1}{2}\big\langle\sigma_t,Hess(f)\sigma_t\big\rangle_{\mathbb{R}^d},$$where $Hess(f)$ denotes the usual Euclidean Hessian matrix of $f \in C^2(\mathbb{R}^d)$. In this spirit, it seems natural to consider (\ref{SNLFPKE}) as a stochastic equation with state space $\mathcal{SP}$ instead of $\mathbb{R}^d$ as for (\ref{StochEq}) and (\ref{P-CE}) as the corresponding linear Fokker-Planck-type equation on $\mathcal{SP}$.

\end{rem}
 By the above derivation, for any subprobability solution process $(\mu_t)_{t \leq T}$ to (\ref{SNLFPKE}) the curve $(\Gamma_t)_{t \leq T}$, $\Gamma_t := \mathbb{P}\circ \mu_t^{-1}$ in $\mathcal{P}(\mathcal{SP})$ solves (\ref{P-CE}) in the sense of the following definition.
\begin{dfn}
	A weakly continuous curve $(\Gamma_t)_{t \leq T} \subseteq \mathcal{P}(\mathcal{SP})$ is a \textit{solution to (\ref{SNLFPKE})}, if the integrability condition
		\begin{align} \label{2_int_SP-FPKE}
		\int_0^T \int_{\mathcal{SP}}||b(t,\mu)||_{L^1(\mathbb{R}^d,\mathbb{R}^d;\mu)}+||a(t,\mu)||_{L^1(\mathbb{R}^d,\mathbb{R}^{d^2};\mu)}+ ||\sigma(t,\mu)||^2_{L^2(\mathbb{R}^d,\mathbb{R}^{d\times d_1};\mu)}d\Gamma_t(\mu)dt < \infty
	\end{align}
	 is fulfilled and for each $F \in \mathcal{F}C^2_b(\mathcal{H})$, \ref{P-FPKE eq} holds for each $t \in [0,T]$.
\end{dfn}
\subsubsection*{Transferring (\ref{SNLFPKE}) and (\ref{P-FPKE}) to $\ell^2$}Reminiscent to the deterministic case, we use the global chart $H: \mathcal{SP}\to \ell^2$ to introduce auxiliary equations on $\ell^2$ and the space of measures on $\ell^2$, respectively, as follows. Again, we use the notation
$$A_t:= \bigg\{\mu \in \mathcal{SP}: \int_{\mathbb{R}^d}|a_{ij}(t,\mu,x)|+|b_i(t,\mu,x)|d\mu(x) < \infty\,\,\forall \,1\leq i,j \leq d\bigg\}, \quad t \in [0,T].$$
For $i,j \geq 1$, $\alpha \leq d_1$, define the measurable coefficients $B_i$ for $(t,\mu)$ such that $\mu \in A_t$, and $\Sigma^{\alpha}_i$ and $A_{ij}$ on $[0,T]\times \mathcal{SP}$ by
\begin{align*}
B_i(t,\mu) &:= \int_{\mathbb{R}^d}\mathcal{L}_{t,\mu}h_i(x)d\mu(x), \quad (t,\mu)\in [0,T]\times A_t,\\
\Sigma^{\alpha}_i(t,\mu) &:= \int_{\mathbb{R}^d}\sigma^{\alpha}(t,\mu,x)\cdot \nabla h_i(x)d\mu(x), \\ \Sigma_i(t,\mu) &:= \big( \Sigma^{\alpha}_i(t,\mu)\big)_{\alpha \leq d_1} , \\ A_{ij}(t,\mu) &:= \big \langle \Sigma_i, \Sigma_j \big \rangle_{d_1} (t,\mu),
\end{align*}
and set 
$$B := (B_i)_{i \geq 1},\, \Sigma := (\Sigma^{\alpha}_i)_{\alpha \leq d_1, i \geq 1}, \, A := (A_{ij})_{i,j \geq 1}.$$ Now, transferring to $\ell^2$, define $\bar{B}, \bar{\Sigma}$ and $\bar{A}_{ij}$ on $[0,T] \times \ell^2$ component-wise via
\begin{align*}
\bar{B}_i(t,z) := 
\begin{cases}
B_i(t,H^{-1}(z))&, z \in H(A_t) \\
0 &, \text{else}
\end{cases},
\end{align*},
\begin{align*}
\bar{\Sigma}^{\alpha}_i(t,z) := 
\begin{cases}
\Sigma_i^{\alpha}(t,H^{-1}(z))&, z \in H(\mathcal{SP}) \\
0 &, z \in \ell^2 \backslash H(\mathcal{SP})
\end{cases},
\end{align*}
$$\bar{\Sigma}_i(t,z) := \big(\bar{\Sigma}^{\alpha}_i(t,z)\big)_{\alpha \leq d_1},$$
$$\bar{A}_{ij}(t,z) := \big \langle\bar{\Sigma}_i,\bar{\Sigma}_j\big \rangle_{d_1}(t,z).$$
$\bar{B}$ and $\bar{\Sigma}^{\alpha}$ are $\ell^2$-valued, since for $z = H(\mu)$
$$|\bar{B}_i(t,z)| \leq \int_{\mathbb{R}^d}|\mathcal{L}_{t,\mu}h_i(x)|d\mu(x) \leq C2^{-i},$$
where $C = C(a,b,d)$ is a finite constant independent of $t,z$ and $i \geq 1$. A similar argument is valid for each $\bar{\Sigma}^{\alpha}$. Each $\bar{B}_i$ and $\bar{\Sigma}^{\alpha}_i$ is product-measurable with respect to the $\ell^2$-topology due to the measurability of $B$ and $\Sigma^{\alpha}$. Reminiscent to (\ref{Rinfty-CE}) in the previous section, we associate to (\ref{P-FPKE}) the FPK-equation on $\ell^2$
\begin{equation}\label{ell2-FPKE}\tag{$\ell^2$-FPK}
\partial_t \bar{\Gamma}_t = -\bar{\nabla}\cdot(\bar{B}(t,z)\bar{\Gamma}_t)+\partial^2_{ij}(\bar{A}_{ij}(t,z)\bar{\Gamma}_t),
\end{equation}which we understand in the sense of the following definition, with $\bar{\nabla}$ as in (\ref{Def nabla gradient}). Subsequently, we denote by $\mathcal{F}C^2_b(\ell^2)$ the set of all maps $\bar{F}: \ell^2 \to \mathbb{R}$ of type $\bar{F} = f \circ \pi_n$ for $n \geq 1$ and $f \in C^2_b(\mathbb{R}^n)$. Also, set
$$D^2\bar{F}_{ij} := 
\begin{cases}
(\partial^2_{ij}f) \circ \pi_n&,  i,j \leq n \\
0&,  \text{ else}.
\end{cases}$$Consequently, both summands in (\ref{14}) contain only finitely many non-trivial summands. 
\begin{dfn}\label{Def sol ell2-FPKE}
	A weakly continuous curve $(\bar{\Gamma}_t)_{t \leq T}\subseteq \mathcal{P}(\ell^2)$ is a \textit{solution to} (\ref{ell2-FPKE}), if it fulfills the integrability condition
		\begin{equation}\label{2_global_int_l2-FPKE}
		\int_0^T\int_{\ell^2}|\bar{B}_i(t,z)| + |\bar{A}_{ij}(t,z)|d\bar{\Gamma}_tdt < \infty,\quad \forall \,i,j \geq 1,
	\end{equation}
	and for any $\bar{F} \in \mathcal{F}C^2_b(\ell^2)$, $\bar{F}:= f\circ \pi_n$,
	\begin{equation}\label{14}
	\int_{\ell^2}\bar{F}(z)d\bar{\Gamma}_t(z) = \int_{\ell^2}\bar{F}(z)d\bar{\Gamma}_0(z)+\int_0^t \int_{\ell^2}\bar{\nabla} \bar{F}(z) \cdot \bar{B}(s,z)+\frac{1}{2}D^2\bar{F}(z):\bar{A}(s,z)d\bar{\Gamma}_s(z)ds.
	\end{equation}holds for each $t \leq T$. 
\end{dfn}

\subsection{Main Result: Stochastic case}
The main result of this section is the following superposition principle for solutions to (\ref{SNLFPKE}) and (\ref{P-FPKE}), which generalizes Theorem \ref{main thm det case} to stochastically perturbed equations.
\begin{theorem}\label{main thm stoch case}
Let $\sigma$ be bounded on $[0,T]\times \mathcal{SP}\times \mathbb{R}^d$.	Let $(\Gamma_t)_{t \leq T}$ be a weakly continuous solution to (\ref{P-FPKE}). Then, there exists a complete filtered probability space $(\Omega,\mathcal{F}, (\mathcal{F}_t)_{t \leq T}, \mathbb{P})$, an adapted $d_1$-dimensional Wiener process $W = (W_t)_{t \leq T}$ and a $\mathcal{SP}$-valued adapted vaguely continuous process $\mu=(\mu_t)_{t \leq T}$ such that $(\mu,W)$ solves (\ref{SNLFPKE}) and
	$$\mathbb{P}\circ \mu_t^{-1} = \Gamma_t$$
	holds for each $t \in [0,T]$.\\
	Moreover, if $\Gamma_0$ is concentrated on $\mathcal{P}$, i.e. $\Gamma_0(\mathcal{P}) = 1$, then the paths $t \mapsto \mu_t(\omega)$ are $\mathcal{P}$-valued for $\mathbb{P}$-a.e. $\omega \in \Omega$ and hence even weakly continuous.
\end{theorem}
As in the proof of \ref{main thm det case}, we proceed in three steps. Since parts of the proof are technically more involved than in the deterministic case, we first present the ingredients of each step and afterwards state the proof of Theorem \ref{main thm stoch case} as a corollary.\\
\\
\textbf{Step 1: From (\ref{P-FPKE}) to (\ref{ell2-FPKE}):} 
\begin{lem}\label{Lem of Step1 stochCase}
	For any solution $(\Gamma_t)_{t \leq T}$ to (\ref{P-FPKE}), the curve $\bar{\Gamma}_t = \Gamma_t \circ H^{-1}$ is a solution to (\ref{ell2-FPKE}).
\end{lem}
\begin{proof}
Clearly, $t \mapsto \bar{\Gamma}_t$ is a weakly continuous curve in $\mathcal{P}(\ell^2)$ due to the continuity of $H: \mathcal{SP} \to \ell^2$. \eqref{2_global_int_l2-FPKE} holds, since $t \mapsto \Gamma_t$ fulfills \eqref{2_int_SP-FPKE} and since $\sigma$ is bounded. Moreover, we have for $s,t \leq T$, $\bar{F} = f \circ \pi_n \in \mathcal{F}C^2_b(\ell^2)$ and $F: \mu \mapsto f\big(\mu(h_1),\dots,\mu(h_n)\big)$
\begin{align*}
& \int_{\ell^2}\bar{\nabla} \bar{F}(z) \cdot \bar{B}(s,z)+\frac{1}{2}D^2\bar{F}(z):\bar{A}(s,z)d\bar{\Gamma}_s(z) \\&=  \int_{\mathcal{SP}}\sum_{k=1}^n(\partial_kf)\big(\mu(h_1),\dots,\mu(h_n)\big)B_k(s,\mu) + \sum_{\alpha=1}^{d_1} \frac{1}{2}\sum_{k,l=1}^n(\partial_{kl}f)\big(\mu(h_1),\dots,\mu(h_n)\big)\Sigma_k^{\alpha}(s,\mu)\Sigma_l^{\alpha}(s,\mu)d\Gamma_s(\mu) \\&=
\int_{\mathcal{SP}}\big \langle \nabla^{\mathcal{SP}}F(\mu),b(s,\mu)+a(s,\mu)\nabla \big \rangle_{L^2(\mu)}+ \frac{1}{2}\sum_{\alpha=1}^{d_1}Hess(F)\big(\sigma^{\alpha}(s,\mu), \sigma^{\alpha}(s,\mu)\big)d \Gamma_s(\mu)
\end{align*}and likewise
\begin{equation*}
\int_{\ell^2}\bar{F}(z)d\bar{\Gamma}_t = \int_{\mathcal{SP}}F(\mu)d\Gamma_t.
\end{equation*}Comparing with (\ref{P-FPKE eq}), the statement follows.
\end{proof}
	
\textbf{Step 2: From (\ref{ell2-FPKE}) to the martingale problem ($\ell^2$-MGP):} We introduce a martingale problem on $\ell^2$, which is related to (\ref{ell2-FPKE}) in the sense of Remark \ref{Rem easy connect stoch ell2} below and is, roughly speaking, the stochastic analogue to (\ref{Rinfty-ODE}) from the previous section. Recall the notation $e_t$ for the projection $e_t: C_T\ell^2 \to \ell^2$, $e_t: \gamma \mapsto \gamma_t$ for $t \leq T$.
\begin{dfn}\label{Def ell2-MGP sol}
	A measure $\bar{Q} \in \mathcal{P}(C_T\ell^2)$ is a \textit{solution to the $\ell^2$-martingale problem ($\ell^2$-MGP)}, provided
		\begin{equation}
		\int_{C_T\ell^2}\int_0^T|\bar{B}_i(t,e_t)|+|\bar{A}_{ij}(t,e_t)|dtd\bar{Q} < \infty,\quad i,j \geq 1,
	\end{equation}
	and
	\begin{equation}\label{15}
	\bar{F}\circ e_t - \int_0^t \bar{\nabla} \bar{F} \circ e_s \cdot \bar{B}(s,e_s)+\frac{1}{2}D^2\bar{F}\circ e_s : \bar{A}(s,e_s)ds
	\end{equation}is a $\bar{Q}$-martingale on $C_T\ell^2$ with respect to the natural filtration on $C_T\ell^2$ for any $\bar{F} \in \mathcal{F}C^2_b(\ell^2)$.
\end{dfn}
\begin{rem}\label{Rem easy connect stoch ell2}
	By construction, any such solution $\bar{Q}$ induces a weakly continuous solution $(\bar{\Gamma}_t)_{t \leq T}$ to (\ref{ell2-FPKE}) via $\bar{\Gamma}_t := \bar{Q} \circ e_t^{-1}$. Indeed, this is readily seen by integrating (\ref{15}) with respect to $\bar{Q}$ and Fubini's theorem.
\end{rem}
In view of Proposition \ref{Prop SPprincip ell2} below, we extend the coefficients $\bar{B}_i, \bar{\Sigma}^{\alpha}_i$ (and hence also $\bar{A}_{ij}$) to $\mathbb{R}^{\infty}$ via
$$\bar{B}_i := 0 =: \bar{\Sigma}^{\alpha}_i \text{ on }[0,T]\times \mathbb{R}^{\infty} \backslash \ell^2.$$
We still use the notation $\bar{B}$, $\bar{\Sigma}^{\alpha}$ and $\bar{A}$ and note that they are $\mathcal{B}([0,T])\otimes \mathcal{B}(\mathbb{R}^{\infty})/\mathcal{B}(\mathbb{R}^{\infty})$-measurable due to Remark \ref{Rem meas P SP} below. Due to the same remark, we may regard any solution $(\bar{\Gamma}_t)_{t \leq T}$ to (\ref{ell2-FPKE}) as a solution to a FPK-equation on $\mathbb{R}^{\infty}$ by considering (\ref{ell2-FPKE}) with the extended coefficients and test functions $\bar{F} \in \mathcal{F}C^2_b(\ell^2)$ extended to $\mathbb{R}^{\infty}$ by considering $\pi_n$ on $\mathbb{R}^{\infty}$ instead of $\ell^2$. Similarly, the formulation of the martingale problem ($\ell^2$-MGP) as in Definition \ref{Def ell2-MGP sol} extends to $\mathbb{R}^{\infty}$ in the sense that a measure $\bar{Q} \in \mathcal{P}(C_T\mathbb{R}^{\infty})$ is understood as a solution, provided the process (\ref{15}) is a $\bar{Q}$-martingale on $C_T\mathbb{R}^{\infty}$ with respect to the natural filtration for each $\bar{F} = f \circ \pi_n: \mathbb{R}^{\infty} \to \mathbb{R}$ as above.
\begin{rem}\label{Rem meas P SP}
	We recall that $\ell^2 \in \mathcal{B}(\mathbb{R}^{\infty})$ and $\mathcal{B}(\ell^2) = \mathcal{B}(\mathbb{R}^{\infty})_{\upharpoonright \ell^2}$. In particular, any probability measure $\bar{\Gamma} \in \mathcal{P}(\ell^2)$ uniquely extends to a Borel probability measure on $\mathbb{R}^{\infty}$ via $\bar{\Gamma}(A) := \bar{\Gamma}(A \cap \ell^2)$, $A \in \mathcal{B}(\mathbb{R}^{\infty})$.
\end{rem}
 We shall need the following superposition principle from  \cite{TrevisanPhD}, which lifts a solution to a FPK-equation on $\mathbb{R}^{\infty}$ to a solution of the associated martingale problem. Note that in \cite{TrevisanPhD}, the author assumes an integrability condition of order $p>1$ instead of $p=1$ as in \eqref{2_global_int_l2-FPKE} in order to essentially reduce the proof to the corresponding finite-dimensional result, see \cite[Thm.2.14]{TrevisanPhD}, which requires such a higher order integrability. However, since the latter result was extended to the case of an $L^1$-integrability condition by the same author \cite[Thm.2.5]{trevisan2016}, it is easy to see that also the infinite-dimensional result \cite[Thm.7.1]{TrevisanPhD} holds for solutions with $L^1$-integrability as in \eqref{2_global_int_l2-FPKE}.

\begin{prop}\label{Prop SPprincip ell2}
	[Superposition principle on $\mathbb{R}^{\infty}$, Thm.7.1. \cite{TrevisanPhD} For any weakly continuous solution $(\bar{\Gamma}_t)_{t \leq T} \subseteq \mathcal{P}(\mathbb{R}^{\infty})$ to the $\mathbb{R}^{\infty}$-extended version of (\ref{ell2-FPKE}), there exists  $\bar{Q} \in \mathcal{P}(C_T\mathbb{R}^{\infty})$, which solves the $\mathbb{R}^{\infty}$-extended version of ($\ell^2$-MGP) such that $\bar{Q}\circ e_t^{-1} = \bar{\Gamma}_t$ for each $t \in [0,T]$.
\end{prop}
Moreover, we have the following consequence for the solutions we are interested in. Note that paths $t \mapsto z_t \in H(\mathcal{SP})$ are continuous with respect to the product topology if and only if they are $\ell^2$-continuous. Hence, we may use the notation $C_TH(\mathcal{SP})$ unambiguously and consider it as a subset of either $C_T\mathbb{R}^{\infty}$ or $C_t\ell^2$. Since $H(\mathcal{SP}) \subseteq \ell^2$ is closed even with respect to the product topology, $C_TH(\mathcal{SP})$ belongs to $\mathcal{B}(C_T\ell^2)$ and $\mathcal{B}(C_T\mathbb{R}^{\infty})$. 
\begin{lem}
	If in the situation of the previous proposition each $\bar{\Gamma}_t$ is concentrated on the Borel set $H(\mathcal{SP}) \subseteq \mathbb{R}^{\infty}$, then $\bar{Q}$ is concentrated on continuous curves in $H(\mathcal{SP})$. In particular, in this case $\bar{Q}$ may be regarded as an element of $\mathcal{P}(C_T\ell^2)$ and a solution to the martingale problem ($\ell^2$-MGP) as in Definition \ref{Def ell2-MGP sol}.
\end{lem}
\begin{proof}
	The closedness of $H(\mathcal{SP}) \subseteq \mathbb{R}^{\infty}$ yields
	$$\bar{Q}(C_TH(\mathcal{SP})) = \bar{Q}\bigg(\underset{q \in [0,T]\cap \mathbb{Q}}{\bigcap}\{e_q \in H(\mathcal{SP})\}\bigg) = 1,$$
	due to $\bar{Q}\circ e_t^{-1} = \bar{\Gamma}_t$ for each $ t \leq T$. By the observation above this lemma, it follows
	$$\mathcal{B}(C_T\ell^2)_{\upharpoonright C_TH(\mathcal{SP})} \subseteq \mathcal{B}(C_T\mathbb{R}^{\infty})_{\upharpoonright C_TH(\mathcal{SP})} \subseteq \mathcal{B}(C_T\mathbb{R}^{\infty})$$and we can therefore consider $\bar{Q}$ as a probability measure on $\mathcal{B}(C_T\ell^2)$ via 
	$$\bar{Q}(A) := \bar{Q}\big(A\cap C_TH(\mathcal{SP})\big),\,\, A \in \mathcal{B}(C_T\ell^2) $$with mass on $C_TH(\mathcal{SP})$. It is clear that this measure fulfills Definition \ref{Def ell2-MGP sol}.
\end{proof}
Hence, subsequently we may regard to $\bar{Q}$ as in Proposition \ref{Prop SPprincip ell2} as a solution to ($\ell^2$-MGP) on either $\mathbb{R}^{\infty}$ or $\ell^2$ without differing the notation. Recall the notation $p_i: \ell^2 \to \mathbb{R}$, $p_i(z) = z_i.$ 
\begin{lem}\label{Lem martingale and covar}
	Let $\bar{Q}$ be a solution to the martingale problem ($\ell^2$-MGP) on $\ell^2$. Then, for any $i \geq 1$, the process
	\begin{equation}\label{16}
	M_i(t) := p_i \circ e_t - p_i \circ e_0 - \int_0^t \bar{B}_i(s,e_s)ds
	\end{equation}
	is a real-valued, continuous $\bar{Q}$-martingale on $C_T\ell^2$ with respect to the canonical filtration. The covariation $\langle \langle M_i, M_j \rangle \rangle $ of $M_i$ and $M_j$ is $\bar{Q}$-a.s. given by
	
	\begin{equation}\label{17}
	\langle\langle M_i, M_j \rangle\rangle_t = \int_0^t \bar{A}_{ij}(s,e_s)ds, \,\, t\in [0,T].
	\end{equation}
\end{lem}
\begin{proof}
	For $i,j \geq 1$, let $n \geq \text{max}(i,j)$, consider $p^n_i: \mathbb{R}^n \to \mathbb{R}$, $p^n_i(x) = x_i$ and let
	$$\bar{F}^n_i : \ell^2 \to \mathbb{R}, \,\, \bar{F}^n_i(z) = p^n_i \circ \pi_n(z).$$ Note that $\bar{F}^n_i = p_i$ on $\ell^2$, independent of $n \geq \text{max}(i,j)$. For $k \geq 1$, introduce the stopping time $\sigma_k :=\text{inf}\{t \in [0,T]: ||e_t||_{\ell^2} \geq k \}$ with respect to the canonical filtration on $C_T\ell^2$. Clearly, $\sigma_k \nearrow +\infty$ pointwise. Consider $\eta_k \in C^{2}_c(\mathbb{R}^n)$ such that $\eta_k(x) = 1$ for  $|x| \leq k+1$.\\ Since $\partial_{k}p^n_i = \delta_{ki}$ and $\partial_{kl}p^n_i = 0$ for $k,l \leq n$, we have
	$$M_i(t) = \bar{F}^n_i\circ e_t - \int_0^t \bar{\nabla} \bar{F}^n_i \circ e_s \cdot \bar{B}(s,e_s)+\frac{1}{2}D^2\bar{F}^n_i\circ e_s : \bar{A}(s,e_s)ds$$ and, setting $\bar{F}^{n,k}_i := (\eta_kp^n_i) \circ \pi_n \in \mathcal{F}C^2_b(\ell^2)$, 
	$$M_i(\sigma_k \wedge t) = \bar{F}^{n,k}_i\circ e_{t\wedge \sigma_k} - \int_0^{t\wedge \sigma_k} \bar{\nabla} \bar{F}^{n,k}_i \circ e_s \cdot \bar{B}(s,e_s)+\frac{1}{2}D^2\bar{F}^{n,k}_i\circ e_s : \bar{A}(s,e_s)ds.$$Since the latter is a continuous $\bar{Q}$-martingale for each $k \geq 1$, it follows that $M_i$ is a continuous local $\bar{Q}$-martingale. Concerning (\ref{17}), it suffices to prove that for any $\bar{F} \in \mathcal{F}C^2_b(\mathcal{\ell}^2)$, $\bar{F} = f \circ\pi_n$, we have
	 \begin{equation}\label{Covar aux}
	\langle \langle M^{\bar{F}} \rangle \rangle_t = \int_0^t \big \langle \bar{\nabla}\bar{F}(e_s), \bar{A}(s,e_s)\bar{\nabla}\bar{F}(e_s)ds \big \rangle_{\ell^2}ds, 
	\end{equation}with
	$$M^{\bar{F}}_t := \bar{F}\circ e_t - \int_0^t   \bar{\nabla}\bar{F}(e_s) \cdot \bar{B}(s,e_s) + \frac{1}{2}D^2\bar{F}(e_s): \bar{A}(s,e_s)ds.$$Indeed, from here (\ref{17}) follows by considering (\ref{Covar aux}) for $\bar{F}^{n,k}_i$, localization of the local martingale $M_i$ and polarization for the quadratic (co-)variation. Concerning (\ref{Covar aux}), it is standard (cf. \cite[p.73,74]{stroock_1987}) to use Itô's product rule to obtain that
	$$t\mapsto (M^{\bar{F}}_t)^2 - \int_0^t \bar{\mathbf{L}}_s^{(2)}\bar{F}^2(e_s)-2\bar{F}(e_s)\bar{\mathbf{L}}_s^{(2)}\bar{F}(e_s)ds $$
	is a continuous $\bar{Q}$-martingale on $C_T\ell^2$, where we denote by $\bar{\mathbf{L}}_t^{(2)}\bar{F}(e_s)$ the integrand of the integral term in the definition of $M^{\bar{F}}$. A straightforward calculation yields
	$$\int_0^t \bar{\mathbf{L}}_s^{(2)}\bar{F}^2(e_s)-2\bar{F}(e_s)\bar{\mathbf{L}}_s^{(2)}\bar{F}(e_s)ds = \int_0^t \big \langle \bar{\nabla}\bar{F}(e_s), \bar{A}(s,e_s)\bar{\nabla}\bar{F}(e_s) \big \rangle_{\ell^2}ds,$$
	which completes the proof.
\end{proof}
We summarize the results of this step in the following proposition.
\begin{prop}\label{Prop summary Step2}
	Let $(\bar{\Gamma}_t)_{t \leq T}$ be a weakly continuous solution to (\ref{ell2-FPKE}) such that $\bar{\Gamma}_t(H(\mathcal{SP})) = 1$ for each $t \in [0,T]$. Then, there exists a solution $\bar{Q} \in \mathcal{P}(C_T\ell^2)$ to the martingale problem ($\ell^2$-MGP) such that $\bar{Q}$ is concentrated on $C_TH(\mathcal{SP})$ with $\bar{Q}\circ e_t^{-1} = \bar{\Gamma}_t$ for each $t \in [0,T]$. Further, the results of Lemma \ref{Lem martingale and covar} apply to $\bar{Q}$. 
\end{prop}

\textbf{Step 3: From ($\ell^2$-MGP) to (\ref{SNLFPKE}):} For a given solution $\bar{Q} \in \mathcal{P}(C_T\ell^2)$ to ($\ell^2$-MGP), set
$$\mathcal{C}:= \mathcal{B}(C_T\ell^2) \bigvee \mathcal{N}_{\bar{Q}}$$
and
$$\mathcal{C}_t := \sigma(e_s, s\leq t) \bigvee \mathcal{N}_{\bar{Q}}$$
for $t \leq T$, where $\mathcal{N}_{\bar{Q}}$ denotes the collection of all subsets of sets $N \in \mathcal{B}(C_T\ell^2)$ with $\bar{Q}(N) = 0$. Of course, $\mathcal{C}$ and $\mathcal{C}_t$ depend on $\bar{Q}$, but we suppress this dependence in the notation. Without further mentioning, we understand such $\bar{Q}$ as extended to $\mathcal{C}$ in the canonical way. Then, $(C_T\ell^2,\mathcal{C},(\mathcal{C}_t)_{t \leq T}, \bar{Q})$ is a complete filtered probability space. Clearly, $(t,\gamma) \mapsto \bar{\Sigma}(t,e_t(\gamma))$ is $\mathcal{C}_t$-progressively measurable from $[0,T]\times C_T\ell^2$ to $L(\mathbb{R}^{d_1},\ell^2)$, the space of bounded linear operators from $\mathbb{R}^{d_1}$ to $\ell^2$.

\begin{rem}\label{Rem extend prob space}
	We extend $(C_T\ell^2,\mathcal{C},(\mathcal{C}_t)_{t \leq T}, \bar{Q})$ as follows. Let $(\Omega', \mathcal{F}'', (\mathcal{F}''_t)_{t \leq T}, P)$ be a complete filtered probability space with a real-valued $\mathcal{F}''_t$-Wiener process $\beta$ on it, define
	\begin{equation*}
	\Omega := C_T\ell^2 \otimes \underset{l \geq 1}{\bigotimes}\Omega', \,\, \mathcal{F}' := \mathcal{C}\otimes \underset{l \geq 1}{\bigotimes}\mathcal{F}'', \,\, \mathcal{F}_t' := \mathcal{C}_t \otimes \underset{l \geq 1}{\bigotimes}\mathcal{F}''_t, \,\, \mathbb{P}' := \bar{Q}\otimes \underset{l \geq 1}{\bigotimes}P,
	\end{equation*}let $\mathcal{F}$ and $\mathcal{F}_t$ be the $\mathbb{P}'$-completion of $\mathcal{F}'$ and $\mathcal{F}_t'$, respectively, and denote the canonical extension of $\mathbb{P}'$ to $\mathcal{F}$ by $\mathbb{P}$. Further, we denote the Wiener process $\beta$ on the $i$-th copy of $\Omega'$ by $\beta_i$ and extend each $\beta_i$ to $\Omega$ by $\beta_i(\omega) := \beta_i(\omega_i)$ for $\omega = \gamma \times (\omega_i)_{i \geq 1} \in \Omega$. Similarly, we extend each projection $e_t$ from $C_T\ell^2$ to $\Omega$ via $e_t(\omega) := e_t(\gamma)$ for $\omega$ as above, but keep the same notation for this extended process. Obviously, $(e_t)_{t \leq t}$ is a continuous, $\mathcal{F}_t$-adapted process on $\Omega$ and each $\beta_i$ is an $\mathcal{F}_t$-Wiener process on $\Omega$ under $\mathbb{P}$. Moreover, $(e_t)_{t \leq T}$ and $(\beta_i)_{i \geq 1}$ are independent on $\Omega$ with respect to $\mathbb{P}$ by construction. Further, it is clear that the process $M_i$ as in (\ref{16}) is a $\mathbb{P}$-martingale with respect to $\mathcal{F}_t$ for each $i \geq 1$ with covariation as in (\ref{17}) and that $(t,\gamma) \mapsto \bar{\Sigma}(t,e_t(\gamma)) \in L(\mathbb{R}^{d_1},\ell^2)$ is $\mathcal{F}_t$-progressively measurable on $[0,T]\times \Omega$.
\end{rem}

Finally, we need the following result, which is a special case of Theorem 2, \cite{Ondrejat_StochInt_Repr}.
\begin{prop}\label{Prop Step3 Ondrejat appl}
	Let $\bar{Q} \in \mathcal{P}(C_T\ell^2)$ be a solution to the martingale problem ($\ell^2$-MGP). Then, there exists a complete filtered probability space with an adapted $d_1$-dimensional Wiener process $W = (W^{\alpha})_{\alpha \leq d_1}$ and an $\ell^2$-valued adapted continuous process $X = (X_t)_{t \leq T}$ such that the law of $X$ on $C_T\ell^2$ is $\bar{Q}$ and for $i \geq 1$ and $t \in [0,T]$, we have a.s.
	\begin{equation}\label{extra2}
	p_i \circ X_t - p_i \circ X_0 - \int_0^t \bar{B}_i(s,X_s)ds = \sum_{\alpha =1}^{d_1}\int_0^t \bar{\Sigma}^{\alpha}_i(s,X_s)dW_s^{\alpha}
	\end{equation}and the exceptional set can be chosen independent of $t$ and $i$.
\end{prop}
To see this, consider Theorem 2 of \cite{Ondrejat_StochInt_Repr} with $X = \ell^2$, $U_0 = \mathbb{R}^{d_1}$, $D = \{p_i, i \geq 1\}$, the processes $M(p_i)$ given by $M_i$ as in (\ref{16}) on the probability space $\Omega$ of Remark \ref{Rem extend prob space} and 
$$g_s = \bar{\Sigma}(s,e_s) \in L(\mathbb{R}^{d_1}, \ell^2).$$ These choices fulfill all requirements of \cite{Ondrejat_StochInt_Repr}. In this case, the $\ell^2$-valued process $X$ is given by $X_t = e_t$ on $\Omega$. Since all terms in (\ref{extra2}) are continuous in $t$, the exceptional set may indeed by chosen independently of $t \in [0,T]$ and $i \geq 1$.\\ 
\\
The proof of Theorem \ref{main thm stoch case} now follows from the above three step-scheme as follows.\\
\\
\textit{Proof of Theorem \ref{main thm stoch case}:} Let $(\Gamma_t)_{t \leq T}\subseteq \mathcal{P}(\mathcal{SP})$ be a weakly continuous solution to (\ref{P-FPKE}). By Lemma \ref{Lem of Step1 stochCase} of Step 1, the weakly continuous curve of Borel probability measures on $\ell^2$
$$\bar{\Gamma}_t := \Gamma_t \circ H^{-1}, \,\, t\in [0,T]$$
solves (\ref{ell2-FPKE}) and each $\bar{\Gamma}_t$ is concentrated on $H(\mathcal{SP})$. By Proposition \ref{Prop summary Step2} of Step 2, there exists a solution $\bar{Q} \in \mathcal{P}(C_T\ell^2)$ to the martingale problem ($\ell^2$-MGP), which is concentrated on $C_TH(\mathcal{SP})$ such that
$$\bar{Q} \circ e_t^{-1} = \bar{\Gamma}_t, \,\, t\in [0,T].$$
Further, Lemma \ref{Lem martingale and covar} applies to $\bar{Q}$. By Lemma \ref{Lem martingale and covar} and Proposition \ref{Prop Step3 Ondrejat appl} of Step 3, there is a $d_1$-dimensional $\mathcal{F}_t$-adapted Wiener process $W = (W^{\alpha})_{\alpha \leq d_1}$ and an $\mathcal{F}_t$-adapted process $X$ on some complete filtered probability space $(\Omega, \mathcal{F}, (\mathcal{F}_t)_{t \leq T}, \mathbb{P})$, which fulfill (\ref{extra2}) and $X \in C_TH(\mathcal{SP})$ $\mathbb{P}$-a.s. such that $\bar{Q}$ is the law of $X$.\\
Possibly redefining $X$ on a  $\mathbb{P}$-negligible set (which preserves (\ref{extra2}) and its adaptedness, the latter due to the completeness of the underlying filtered probability space), we may assume $X_t(\omega) = H(\mu_t(\omega))$ for some $\mu_t(\omega) \in \mathcal{SP}$ for each $(t,\omega) \in [0,T]\times \Omega$. The continuity of $H^{-1}: H(\mathcal{SP}) \to \mathcal{SP}$ and $t \mapsto X_t(\omega)$ implies vague continuity of 
\begin{equation}\label{19}
t \mapsto \mu_t(\omega) = H^{-1} \circ X_t(\omega)
\end{equation}
for each $\omega \in \Omega$ and $\mathcal{F}_t$-adaptedness of the $\mathcal{SP}$-valued process $(\mu_t)_{t \leq T}$. Considering (\ref{extra2}), $X_t = H(\mu_t)$ and the definition of $\bar{B}$ and $\bar{\Sigma}^{\alpha}_i$, we obtain, recalling $p_i(H(\nu)) = \nu(h_i)$ for each $\nu \in \mathcal{SP}$,
\begin{equation*}
\mu_t(h_i)-\mu_0(h_i)- \int_0^t B_i(s,\mu_s)ds = \sum_{\alpha = 1}^{d_1}\int_0^t \Sigma^{\alpha}_i(s,\mu_s)dW^{\alpha}_s, \,\, t \leq T
\end{equation*}$\mathbb{P}$-a.s. for each $i \geq 1$. From here, it follows by Lemma \ref{Lem aux H} (i) that $(\mu_t)_{t \leq T}$ is a solution to (\ref{SNLFPKE}) as in Definition \ref{Def sol SNLFPKE}. Further, 
$$\mathbb{P}\circ \mu_t^{-1} = (\mathbb{P}\circ X_t^{-1})\circ (H^{-1})^{-1} = \bar{\Gamma}_t \circ (H^{-1})^{-1} = \Gamma_t \circ H^{-1} \circ (H^{-1})^{-1} = \Gamma_t.$$
It remains to prove the final assertion of the theorem. To this end, note that $\Gamma_0(\mathcal{P}) = 1$ implies $\bar{\Gamma}_0(H(\mathcal{P})) = 1$ and hence $\mu_0 \in \mathcal{P}$ $\mathbb{P}$-a.s. with $\mu_0$ as in (\ref{19}). From here, the assertion follows by Lemma \ref{Lem consv of mass stochastic}. \qed
\begin{rem}
The particular type of noise we consider for \eqref{SNLFPKE} was partially motivated by \cite{Coghi2019StochasticNF}, where the natural connection of equations of type \eqref{SNLFPKE} to interacting particle systems with common noise was investigated. Other types of noise terms may be treated in the future, including a possible extension to infinite-dimensional ones. In particular, Proposition \ref{Prop Step3 Ondrejat appl} via \cite[Thm.2]{Ondrejat_StochInt_Repr} in the final step of the proof seems capable of such extensions, since the latter is an infinite-dimensional result.
\end{rem}

\bibliography{Superpos_FPKE_FPKE2_SPd_incl._SNLPFKE_V1} 
\vspace{1cm}
\textit{Marco Rehmeier} Faculty of Mathematics, Bielefeld University, Universitätsstraße 25, 33615 Bielefeld, Germany\\
\textit{E-mail address: mrehmeier@math.uni-bielefeld.de}

\end{document}